\documentclass[11pt]{article}

\usepackage[a4paper,margin=1in]{geometry}
\usepackage{amsmath,amssymb,amsthm,mathtools}
\usepackage{hyperref}
\hypersetup{hidelinks}
\usepackage{enumitem}
\usepackage[numbers,sort&compress]{natbib}

\newtheorem{theorem}{Theorem}[section]
\newtheorem{proposition}[theorem]{Proposition}
\newtheorem{lemma}[theorem]{Lemma}
\newtheorem{corollary}[theorem]{Corollary}
\newtheorem{remark}[theorem]{Remark}

\newcommand{\dd}{\,\mathrm d}
\newcommand{\E}{\mathbb E}
\newcommand{\Prob}{\mathbb P}
\newcommand{\R}{\mathbb R}
\newcommand{\dom}{\operatorname{dom}}

\title{High-energy asymptotics for finite-interval Schr\"odinger operators with
Gaussian white-noise potential}
\author{
Yingdu Dong\thanks{Affiliation: School of Mathematical Sciences, Fudan University, Shanghai, China. Email: 202131130009@mail.bnu.edu.cn.}
\and
Wenwen Jian\thanks{Affiliation: School of Mathematics, Physics and Statistics, Shanghai Polytechnic University, Shanghai, 201209, China. Email: wwjian@sspu.edu.cn. Funding: NNSFC(11201392).}
\and
Xiaoping Yuan\thanks{Affiliation: School of Mathematical Sciences, Fudan University, Shanghai, China. Email: xpyuan@fudan.edu.cn. Funding:NNSFC(12371189).}
}
\date{}

\begin{document}
\maketitle

\begin{abstract}
We study the one-dimensional Schr\"odinger operator on a fixed interval with
Gaussian white-noise potential,
\[
    H_\omega=-\frac{\dd^2}{\dd x^2}+\rho\dot B_x(\omega),
\]
under Dirichlet boundary conditions. The operator is defined pathwise through
the quasi-derivative realization of Sturm--Liouville operators with
distributional potentials. Let $\lambda_n$ be the Dirichlet eigenvalues,
$\lambda_n^+=\max\{\lambda_n,0\}$, and $k_n=\sqrt{\lambda_n^+}$. For every
finite $p$, we prove the high-energy expansion
\[
    k_n=\frac{n\pi}{L}
    +\frac{\rho}{n\pi}\int_0^L
    \sin^2\left(\frac{n\pi s}{L}\right)\,\dd B_s
    +O_{L^p(\Omega)}(n^{-2}).
\]
Consequently, almost surely, $\lambda_n>0$ for all sufficiently large $n$ and,
for every $\varepsilon>0$,
\[
    k_n=\frac{n\pi}{L}+O(n^{-1+\varepsilon}).
\]
We also obtain first-order eigenfunction asymptotics with explicit Brownian
stochastic-integral corrections. In particular, for the $L^2(0,L)$-normalized
Dirichlet eigenfunction $\varphi_n$, with a fixed sign convention,
\[
    \sup_{0\le x\le L}
    \left|\varphi_n(x)-\sqrt{\frac{2}{L}}\sin(k_n x)\right|
    =O(n^{-1+\varepsilon})
\]
almost surely. The proofs use stochastic Pr\"ufer coordinates, stochastic
Volterra expansions, the Burkholder--Davis--Gundy inequality, and a
Borel--Cantelli argument. The estimates provide a first step toward KAM-type small-divisor analysis for
Hamiltonian PDEs with white-noise spatial potentials.
\end{abstract}

\section{Introduction}

Schr\"odinger operators with random potentials form one of the basic mathematical models for quantum motion in disordered media. Singular random potentials naturally appear when a disordered medium is observed on a scale much larger than its correlation length. Let
\(
    V=(V(x))_{x\in\R}
\)
be a real-valued, mean-zero, stationary random field on \(\R\), defined on a probability space, and assume that \(\E|V(0)|^2<\infty\). Here a random field means a jointly measurable family of random variables indexed by the spatial variable, and stationarity means that its finite-dimensional distributions are invariant under spatial translations. The square-integrability assumption ensures that the covariance function
\[
    R(r):=\E[V(r)V(0)]
\]
is well defined; by stationarity, \(\E[V(x)V(y)]=R(x-y)\). Assume further that \(R\in L^1(\R)\). For the rescaled field
\[
    V_\varepsilon(x):=\varepsilon^{-1/2}V(x/\varepsilon),
    \qquad 0<\varepsilon\ll1,
\]
one has
\[
    \E[V_\varepsilon(x)V_\varepsilon(y)]
    =\varepsilon^{-1}R\left(\frac{x-y}{\varepsilon}\right).
\]
Thus, by the elementary approximate-identity argument, the covariance kernels converge in the sense of distributions on \(\R^2\) to a multiple of the delta kernel:
\[
    \varepsilon^{-1}R\left(\frac{x-y}{\varepsilon}\right)
    \longrightarrow
    \left(\int_\R R(r)\,\dd r\right)\delta(x-y).
\]
Equivalently, for test functions \(\varphi,\psi\in C_c^\infty(\R)\),
\[
\begin{aligned}
    \E\left[\left(\int_\R V_\varepsilon(x)\varphi(x)\,\dd x\right)
    \left(\int_\R V_\varepsilon(y)\psi(y)\,\dd y\right)\right]
    &\longrightarrow
    \left(\int_\R R(r)\,\dd r\right)
    \int_\R \varphi(x)\psi(x)\,\dd x .
\end{aligned}
\]
This identifies the limiting covariance structure. Under suitable additional mixing and moment assumptions, the corresponding convergence of the random fields is obtained by applying a functional central limit theorem to the additive functional
\[
    W_\varepsilon(x):=\int_0^x V_\varepsilon(s)\,\dd s
    =\sqrt{\varepsilon}\int_0^{x/\varepsilon}V(r)\,\dd r .
\]
The classical i.i.d. prototype is Donsker's invariance principle; see, for example, Billingsley \cite{Billingsley1999}. For continuous-parameter stationary processes, invariance principles for additive functionals under strong mixing-type assumptions go back to Davydov \cite{Davydov1968}; related strong-mixing functional central limit theorems are discussed in \cite{DoukhanMassartRio1994}. For spatial integrals of stationary mixing random fields, see \cite{KampfSpodarev2018}. Under such assumptions one obtains, on compact intervals in the spatial variable,
\[
    \left(\sqrt{\varepsilon}\int_0^{x/\varepsilon}V(r)\,\dd r\right)_{x\ge0}
    \Longrightarrow
    \sigma_{\mathrm{eff}}\widetilde B_x,
    \qquad
    \sigma_{\mathrm{eff}}^2=\int_\R R(r)\,\dd r,
\]
where \(\widetilde B\) is a standard Brownian motion. Equivalently,
\[
    W_\varepsilon(x):=\int_0^x V_\varepsilon(s)\,\dd s
    \Longrightarrow \sigma_{\mathrm{eff}}\widetilde B_x,
\]
and therefore \(V_\varepsilon\) converges to \(\sigma_{\mathrm{eff}}\dot{\widetilde B}\) in the sense of random distributions. Hence Gaussian white noise may be viewed as the limiting model for a centered spatial disorder whose correlation length tends to zero.

This limiting picture motivates the one-dimensional white-noise Schr\"odinger equation
\[
    -y''+\rho \dot B\,y=\lambda y,
\]
where \(B\) is Brownian motion and \(\dot B\) denotes Gaussian white noise. More precisely, Gaussian white noise on \(\R\) is the centered Gaussian random distribution \(\xi\) characterized by
\[
    \E\langle \xi,\varphi\rangle\langle \xi,\psi\rangle
    =\int_\R \varphi(x)\psi(x)\,\dd x,
    \qquad \varphi,\psi\in C_c^\infty(\R),
\]
and it may be realized as the distributional derivative of a two-sided Brownian motion. In the present paper the spatial domain is the fixed interval \((0,L)\). We work on a probability space \((\Omega,\mathcal F,\Prob)\) carrying a standard Brownian motion \(B=(B_x)_{0\le x\le L}\); here \(\Omega\) is the Brownian sample space, and \(\omega\in\Omega\) denotes a realization of the Brownian path. This restriction changes the spectral problem from an infinite-volume one to a finite-interval boundary value problem, but it does not change the local covariance structure inherited from the scaling picture above. For test functions supported in \((0,L)\), the covariance is still given by the \(L^2\)-inner product; equivalently, the potential is still delta-correlated in the spatial variable. The finite-interval model is therefore obtained by restricting the white noise to \([0,L]\), writing its primitive as \(\rho B\), and interpreting the resulting equation through the quasi-derivative formalism described below. Similar scale-separation ideas also appear at the level of equations with large or rapidly oscillating random potentials: one studies the limiting effective models obtained after averaging or rescaling the random medium; see the review \cite{BalGuReview}.

For general background on random Schr\"{o}dinger operators, we refer the reader to \cite{CarmonaLacroix,PasturFigotin,Kirsch2008}. To the best of our knowledge, one of the first rigorous studies of random Schr\"{o}dinger operators with white-noise potential is due to Fukushima and Nakao~\cite{FukushimaNakao}. Since then, such operators have been investigated from several viewpoints, including the density of states, localization, large-volume spectral statistics, and related half-line models; see, for instance, \cite{Thompson1983,DumazLabbe2020,DumazLabbe2023,DumazLabbe2024,DumazLabbe2024Crossover,Minami,RRV}. These works mainly concern infinite-volume, density-of-states, localization, large-volume spectral-statistics, or half-line questions. The problem studied here is different in nature: the interval length is fixed and the high-energy index tends to infinity. The aim is therefore to obtain finite-interval eigenvalue and eigenfunction asymptotics in which both the leading stochastic correction and the probabilistic size of the remainders are explicit.

The pathwise realization of this finite-interval problem relies on the deterministic theory of Sturm--Liouville operators with distributional potentials. We recall the relevant formulation. If a real distribution \(q\) belongs to \(H^{-1}(0,L)\), it can be written as \(q=\sigma'\) with \(\sigma\in L^2(0,L)\), and the differential expression is defined by the quasi-derivative
\[
    y^{[1]}=y'-\sigma y,
    \qquad
    \ell_\sigma y=-(y^{[1]})'-\sigma y^{[1]}-\sigma^2y.
\]
For white noise with strength \(\rho\), the primitive is \(\sigma=\rho B\), and hence \(q=\sigma'=\rho\dot B\) in the distributional sense. Since Brownian paths are continuous on \([0,L]\), one has \(\sigma\in L^2(0,L)\) almost surely, so the deterministic distributional Sturm--Liouville theory applies pathwise; see, for example, \cite{HrynivMykytyuk2006,SavchukShkalikov,HrynivMykytyuk2012}. Therefore, under self-adjoint boundary conditions, the operator has discrete spectrum. The question addressed below is how the delta-correlated disorder is reflected in the high-energy asymptotics of this finite-interval spectrum and its eigenfunctions.

It is important to distinguish this pathwise random-operator problem from the eigenvalue problem for stochastic Hamiltonian systems with boundary conditions studied by Peng \cite{Peng2000}. In Peng's setting, the eigenvalues are deterministic parameters for which a forward-backward stochastic Hamiltonian system admits a non-trivial adapted solution. In the present paper, after fixing a Brownian path one obtains a self-adjoint Sturm--Liouville operator \(H_\omega\), and the eigenvalues \(\lambda_n(\omega)\) are random variables. Thus the high-energy asymptotics below concern random eigenvalues of a finite-interval random operator, rather than deterministic spectral parameters of an adapted FBSDE boundary value problem.

We next clarify the relation between the present work and the existing deterministic theory. The self-adjointness and compact-resolvent properties used below are standard consequences of the theory of Sturm--Liouville operators with distributional potentials. Moreover, eigenvalue asymptotics for broad classes of singular potentials are already known; see, for example, \cite{HrynivMykytyukJFA2006,EGNT}. These deterministic results provide the natural operator-theoretic framework for the white-noise model considered here. However, the specialization to Gaussian white-noise potentials has additional probabilistic structure which is not captured by the deterministic theory alone. In this setting, the primitive of the potential is Brownian motion, and the leading oscillatory corrections can be written explicitly as Brownian stochastic integrals. This makes it possible to use probabilistic tools, in particular the Burkholder--Davis--Gundy inequality, to obtain quantitative $L^p(\Omega)$ estimates and almost-sure estimates for the random remainders. The aim of the present paper is therefore not merely to rederive deterministic high-energy asymptotics in a special case, but to refine them into explicit stochastic Pr\"ufer expansions for the eigenvalues and eigenfunctions of the finite-interval white-noise Schr\"odinger operator. More precisely, the new information is the explicit Brownian stochastic-integral correction, its $L^p(\Omega)$ control, and the resulting almost-sure high-energy bounds.

We now state the main results in a form convenient for reference. Throughout the following statements, $(\lambda_n)_{n\ge1}$ denotes the Dirichlet eigenvalues, listed increasingly, and
\[
    \lambda_n^+:=\max\{\lambda_n,0\},\qquad
    k_n:=\sqrt{\lambda_n^+},\qquad
    K_n:=\frac{n\pi}{L}.
\]
The convention $k_n=\sqrt{\lambda_n^+}$ is harmless at high energy, since, almost surely, $\lambda_n>0$ for all sufficiently large $n$.

\begin{theorem}[Eigenvalue asymptotics]\label{thm:intro-eigenvalue}
Set
\[
    M_n:=\int_0^L \sin^2(K_ns)\,\dd B_s.
\]
Then, for every finite $p\ge1$,
\[
    k_n
    =K_n+\frac{\rho}{n\pi}M_n+O_{L^p(\Omega)}(n^{-2}),
    \qquad n\to\infty.
\]
Moreover,
\[
    \lambda_n
    =K_n^2+\frac{2\rho}{L}M_n+O_{L^p(\Omega)}(n^{-1}).
\]
In particular, using $\sin^2(K_ns)=1/2-\cos(2K_ns)/2$,
\[
    \lambda_n
    =\left(\frac{n\pi}{L}\right)^2
    +\frac{\rho}{L}B_L
    -\frac{\rho}{L}
    \int_0^L\cos\left(\frac{2n\pi s}{L}\right)\,\dd B_s
    +O_{L^p(\Omega)}(n^{-1}).
\]
\end{theorem}

\begin{corollary}[Almost-sure eigenvalue bounds]\label{cor:intro-eigenvalue-as}
For every $\varepsilon>0$,
\[
    k_n=\frac{n\pi}{L}+O(n^{-1+\varepsilon})
    \qquad\text{almost surely}.
\]
Consequently,
\[
    \lambda_n=\left(\frac{n\pi}{L}\right)^2+O(n^\varepsilon)
    \qquad\text{almost surely}
\]
for every $\varepsilon>0$.
\end{corollary}

The proof of Theorem~\ref{thm:intro-eigenvalue} and Corollary~\ref{cor:intro-eigenvalue-as} is obtained from a stochastic Pr\"ufer transformation. After writing $\lambda=k^2$, the Dirichlet eigenvalue condition becomes a phase condition for the lifted Pr\"ufer angle. Its high-energy expansion gives the displayed stochastic correction, and the almost-sure bound follows from the corresponding finite-moment estimates and the Borel--Cantelli lemma.

We next state the corresponding eigenfunction asymptotics. For deterministic $k>0$, let $u_k$ be the Dirichlet initial solution normalized by
\[
    u_k(0)=0,\qquad u_k'(0)=k.
\]
It satisfies the stochastic Volterra equation
\[
    u_k(x)=\sin(kx)+\frac{\rho}{k}
    \int_0^x \sin(k(x-s))u_k(s)\,\dd B_s.
\]

\begin{theorem}[Eigenfunction asymptotics]\label{thm:intro-eigenfunction}
For every finite $p\ge2$,
\[
\left\|
    \sup_{0\le x\le L}\left|
    u_k(x)-\sin(kx)
    -\frac{\rho}{k}
    \int_0^x\sin(k(x-s))\sin(ks)\,\dd B_s
    \right|
\right\|_{L^p(\Omega)}
    \le C_p k^{-2},
    \qquad k\ge k_0.
\]
For the eigenvalue sequence, define the deterministic-frequency stochastic integral
\[
    I_n(x):=\int_0^x
    \sin\left(K_n(x-s)\right)\sin(K_ns)\,\dd B_s.
\]
Then, for every $\varepsilon>0$,
\[
    \sup_{0\le x\le L}
    \left|u_{k_n}(x)-\sin(k_nx)-\frac{\rho}{k_n}I_n(x)\right|
    =O(n^{-2+\varepsilon})
    \qquad\text{almost surely},
\]
and hence
\[
    \sup_{0\le x\le L}|u_{k_n}(x)-\sin(k_nx)|
    =O(n^{-1+\varepsilon})
    \qquad\text{almost surely}.
\]
Furthermore, let $\varphi_n$ be the $L^2(0,L)$-normalized Dirichlet eigenfunction associated with $\lambda_n$, with the sign chosen so that $\varphi_n'(0)>0$. Put
\[
    s_n(x):=\sin(k_nx),
    \qquad
    A_n:=\int_0^L s_n(t)I_n(t)\,\dd t.
\]
Then, for every $\varepsilon>0$,
\[
\begin{aligned}
    \sup_{0\le x\le L}\Bigg|
    \varphi_n(x)-\sqrt{\frac{2}{L}}
    \bigg[
    s_n(x)+\frac{\rho}{k_n}
    \bigg(I_n(x)-\frac{2}{L}A_ns_n(x)\bigg)
    \bigg]\Bigg|
    =O(n^{-2+\varepsilon})
\end{aligned}
\]
almost surely. In particular,
\[
    \sup_{0\le x\le L}
    \left|\varphi_n(x)-\sqrt{\frac{2}{L}}\sin(k_nx)\right|
    =O(n^{-1+\varepsilon})
    \qquad\text{almost surely}.
\]
\end{theorem}

The detailed asymptotic theorems are proved below for Dirichlet boundary conditions. General separated boundary conditions lead to the same stochastic Pr\"ufer equation but require additional boundary phase corrections; this is discussed briefly in the final section.

A further motivation for the present spectral analysis comes from infinite-dimensional KAM theory for Hamiltonian PDEs. In the classical KAM theory for nonlinear Schr\"odinger equations and wave equations, non-resonance of the linear frequencies is often produced by external parameters. A typical example is a convolution-type potential or diagonal Fourier multiplier, whose Fourier coefficients are treated as parameters; see, for instance, \cite{Kuksin1993,KuksinPoeschel1996}. Although such parameters are regarded as randomized, this randomness is mainly a device for selecting a good set of frequencies. The constant-potential case is still more rigid, since it essentially shifts the linear frequencies by the same amount. Both situations differ from the pathwise randomness of the white-noise operator studied here, where the Brownian path itself determines the whole linear spectrum. Another important line of work treats fixed deterministic multiplication potentials, where no adjustable external parameter is available; see, for example, Du and Yuan \cite{DuYuan2006} for nonlinear Schr\"odinger equations with a given potential, and Yuan \cite{Yuan2006WavePotential} for nonlinear wave equations with a prescribed potential. Compared with these deterministic-potential results, the white-noise problem that we ultimately have in mind is expected to be more difficult. The potential is distributional and random, the frequencies and eigenfunctions are generated simultaneously by a single Brownian path, and the small-divisor estimates cannot be reduced to deterministic spectral separation or to the exclusion of a simple external parameter set.

The long-term goal of the present program is to develop KAM theory for Schr\"odinger or wave equations whose spatial potential is Gaussian white noise. In such a problem the natural frequencies are the random eigenvalues of the pathwise operator. Thus small divisors involve random combinations such as
\[
    \lambda_i(\omega)+\lambda_j(\omega)-\lambda_k(\omega)-\lambda_l(\omega),
\]
and the coefficients in the Hamiltonian normal form also depend on the random eigenfunctions. The results proved in this paper give a first step in this direction: they identify the leading stochastic correction to the high-energy eigenvalues and give uniform eigenfunction asymptotics. However, these estimates are not yet sufficient for a full KAM analysis. For example, polynomial lower bounds for fourth-order divisors require higher-order joint eigenvalue expansions, small-ball estimates for the resulting Gaussian or non-Gaussian stochastic corrections, and summable control of the remainders over large families of almost-resonant integer quadruples.

A possible continuation is therefore to combine higher-order stochastic Volterra or Pr\"ufer expansions with probabilistic separation estimates for the random small divisors. In higher-order expansions one expects random fields indexed by the frequency parameter, and the substitution of random eigenfrequencies may lead either to anticipating stochastic integrals or to equivalent deterministic-frequency re-expansions. Clarifying this point, together with the required joint Gaussian estimates and eigenfunction coefficient bounds, is an essential step toward a KAM theorem with genuine white-noise spatial potential. The present paper isolates the spectral asymptotic component of this larger problem.

The paper is organized as follows. Section 2 introduces the white-noise distribution and its primitive, then recalls the quasi-derivative realization and the basic spectral consequences of self-adjointness and compact resolvent. In Section 3, we derive the stochastic Pr\"ufer equations and prove the high-energy expansion of the Pr\"ufer angle. Subsequently, we give the Dirichlet eigenvalue asymptotics with an explicit stochastic correction and the almost-sure estimate. Section 4 proves the eigenfunction expansion and its almost-sure consequences. Finally, in Section 5, we discuss the results of this paper and possible extensions.

\section{Preliminary}
In this section, we recall some standard facts used to formulate the white-noise operator pathwise.

\subsection{White noise as a distributional potential}

Let $(\Omega,\mathcal F,\Prob)$ be the Brownian probability space introduced in the Introduction, carrying a standard Brownian motion $(B_x)_{0\le x\le L}$. Fix $\rho\in\R$ and define
\[
    \sigma_\omega(x)=\rho B_x(\omega),\qquad 0\le x\le L.
\]
For each fixed $\omega$ outside a null set, $\sigma_\omega\in C[0,L]\subset L^2(0,L)$. We define the potential
\[
    q_\omega=\sigma_\omega'
\]
in the sense of distributions. Equivalently, for $\varphi\in C_c^\infty(0,L)$,
\[
    \langle q_\omega,\varphi\rangle
    =-\int_0^L \sigma_\omega(x)\varphi'(x)\,\dd x.
\]
Since $\sigma_\omega\in L^2(0,L)$,
\[
    |\langle q_\omega,\varphi\rangle|
    \le \|\sigma_\omega\|_{L^2(0,L)}\,\|\varphi'\|_{L^2(0,L)},
\]
so $q_\omega\in H^{-1}(0,L)$ almost surely. (Actually, $q_\omega$ belongs to $H^{-s}$ for $s>1/2$, but we do not need this result.)

One can check that the distribution defined above agrees with the usual Gaussian white noise. Indeed, for $\varphi\in C_c^\infty(0,L)$, It\^o integration yields
\begin{equation*}
    \int_0^L \varphi(x)\,\dd B_x
    =-\int_0^L B_x\varphi'(x)\,\dd x.
\end{equation*}
Therefore
\begin{equation*}
    \langle q_\omega,\varphi\rangle=\rho \int_0^L \varphi(x)\,\dd B_x,
\end{equation*}
and hence
\begin{equation*}
    \mathbb E\langle q_\omega,\varphi\rangle=0,
    \qquad
    \mathbb E\bigl[\langle q_\omega,\varphi\rangle\langle q_\omega,\psi\rangle\bigr]
    =\rho^2\int_0^L \varphi(x)\psi(x)\,\dd x.
\end{equation*}

\subsection{Quasi-derivatives and self-adjoint realization}

Throughout this section let $\sigma\in L^2(0,L)$ be real-valued. Define the quasi-derivative
\[
    y^{[1]}=y'-\sigma y.
\]
The regularized Sturm--Liouville expression associated with $q=\sigma'$ is
\[
    \ell_\sigma y=-(y^{[1]})'-\sigma y^{[1]}-\sigma^2 y.
\]
Formally, if all derivatives are classical, then $\ell_\sigma y=-y''+\sigma' y$. The maximal domain is
\[
    \dom T_{\max}=\{y\in L^2(0,L): y,y^{[1]}\in AC[0,L],\ \ell_\sigma y\in L^2(0,L)\}.
\]
For $y,z\in\dom T_{\max}$ one has the Lagrange identity
\[
    (\ell_\sigma y,z)_{L^2(0,L)}-(y,\ell_\sigma z)_{L^2(0,L)}
    =[y,z]_\sigma(L)-[y,z]_\sigma(0),
\]
where
\[
    [y,z]_\sigma(x)=y^{[1]}(x)\overline{z(x)}-y(x)\overline{z^{[1]}(x)}.
\]

Let $\alpha,\beta\in[0,\pi)$ and impose the separated boundary conditions
\begin{align*}
    y(0)\cos\alpha-y^{[1]}(0)\sin\alpha&=0,\\
    y(L)\cos\beta-y^{[1]}(L)\sin\beta&=0.
\end{align*}
Define $H_{\sigma,\alpha,\beta}y=\ell_\sigma y$ on the subdomain of $\dom T_{\max}$ satisfying these two boundary conditions. Then from classical results for quasi-derivative Sturm--Liouville problem, one can obtain the following result.

\begin{theorem}[Self-adjointness and compact resolvent]\label{thm:sa-compact}
Let $\sigma\in L^2(0,L)$ be real-valued. Then $H_{\sigma,\alpha,\beta}$ is self-adjoint in $L^2(0,L)$. It is bounded from below and has compact resolvent. Consequently its spectrum consists only of real eigenvalues of finite multiplicity, accumulating at $+\infty$ only. For separated boundary conditions all eigenvalues are simple.
\end{theorem}

For a detailed proof, we refer the reader to \cite{SavchukShkalikov}. We recall the main mechanism, since it is used to pass to the white-noise case. The Lagrange identity shows that the imposed real separated boundary conditions make the boundary form vanish, hence the corresponding restriction is symmetric. The regular quasi-derivative Sturm--Liouville theory then shows that these separated boundary conditions give a self-adjoint restriction of the maximal operator.

For completeness, let us also indicate why lower semiboundedness and compactness of the resolvent are natural in this setting. In the Dirichlet case the associated closed form is represented by
\[
    \mathfrak t[y]=\int_0^L |y^{[1]}|^2\,\dd x-\int_0^L \sigma^2|y|^2\,\dd x,
\]
with form domain $H_0^1(0,L)$. In one dimension, by the Gagliardo--Nirenberg interpolation inequality,
\[
    \|y\|_\infty^2\le C_L\|y\|_{L^2(0,L)}\|y'\|_{L^2(0,L)}.
\]
Thus, for every $\eta>0$,
\[
    \int_0^L \sigma^2|y|^2\,\dd x
    \le \|\sigma\|_{L^2(0,L)}^2\|y\|_\infty^2
    \le \eta\|y'\|_{L^2(0,L)}^2+C_{\eta,L,\sigma}\|y\|_{L^2(0,L)}^2.
\]
This shows that the negative part of the form is infinitesimally form-bounded with respect to the Dirichlet energy. Consequently the form is closed and bounded from below, and the first representation theorem for closed semibounded forms (equivalently, the Friedrichs construction in this case) produces a lower semibounded self-adjoint operator. After adding a sufficiently large multiple of the $L^2(0,L)$-norm, the form norm is equivalent to the $H^1(0,L)$-norm. Since the embedding $H^1(0,L)\hookrightarrow L^2(0,L)$ is compact, the associated resolvent is compact. The spectral theorem applied to the compact self-adjoint resolvent then gives a purely discrete real spectrum with finite multiplicities, accumulating only at $+\infty$. Analogous estimates, with the usual finite-dimensional boundary corrections, apply to general separated boundary conditions.

We now apply Theorem~\ref{thm:sa-compact} to the case relevant for this paper. For every Brownian path outside a null set, $B(\omega)\in C[0,L]$, and hence $\sigma_\omega=\rho B(\omega)\in L^2(0,L)$. Therefore Theorem~\ref{thm:sa-compact} gives the following pathwise realization of the white-noise operator.
\begin{corollary}
For almost every $\omega$, define the quasi-derivative
\begin{equation*}
    y^{[1]}:=y'-\sigma_\omega y
\end{equation*}
and the regularized differential expression
\begin{equation*}
    \ell_{\sigma_\omega}y
    :=-\bigl(y'-\sigma_\omega y\bigr)'-\sigma_\omega y'
    =-\bigl(y^{[1]}\bigr)'-\sigma_\omega y^{[1]}-\sigma_\omega^2y .
\end{equation*}
Let $\alpha,\beta\in[0,\pi)$, and define $H_{\omega,\alpha,\beta}$ in $L^2(0,L)$ by
\begin{equation*}
    H_{\omega,\alpha,\beta}y=\ell_{\sigma_\omega}y
\end{equation*}
on the domain
\begin{align*}
\operatorname{dom} H_{\omega,\alpha,\beta}:=\bigl\{y\in L^2(0,L):&\ y,y^{[1]}\in \mathrm{AC}[0,L],\ \ell_{\sigma_\omega}y\in L^2(0,L),\nonumber\\
&\ y(0)\cos\alpha-y^{[1]}(0)\sin\alpha=0,\nonumber\\
&\ y(L)\cos\beta-y^{[1]}(L)\sin\beta=0\bigr\}.
\end{align*}
Then $H_{\omega,\alpha,\beta}$ is a self-adjoint operator in $L^2(0,L)$. Moreover, it is bounded from below and has compact resolvent. Consequently its spectrum is real, purely discrete, and may be listed as
\begin{equation*}
    \lambda_1(\omega)\le \lambda_2(\omega)\le\cdots,
    \qquad \lambda_n(\omega)\to +\infty,
\end{equation*}
with multiplicities counted. 
\end{corollary}
\begin{remark}
	Under the usual real separated boundary conditions, the one-dimensional Sturm--Liouville oscillation theory gives simplicity of the eigenvalues.
\end{remark}

Moreover, the compact resolvent directly implies the completeness of the eigenfunctions. This is a standard consequence of the spectral theorem for compact self-adjoint operators.

\begin{theorem}[Completeness of eigenfunctions]
Almost surely, the eigenfunctions of $H_{\omega,\alpha,\beta}$ form an orthonormal basis of $L^2(0,L)$. More precisely, for almost every $\omega$, there exist real eigenvalues
\[
    \lambda_1(\omega)\le\lambda_2(\omega)\le\cdots,
    \qquad \lambda_n(\omega)\to+\infty,
\]
counted with multiplicity, and an orthonormal family of eigenfunctions $(\varphi_n(\omega))_{n\ge1}$ such that, for every $f\in L^2(0,L)$,
\[
    f=\sum_{n=1}^\infty \langle f,\varphi_n(\omega)\rangle\varphi_n(\omega)~~a.s.
\]
in $L^2(0,L)$ and
\[
    \|f\|_{L^2(0,L)}^2=\sum_{n=1}^\infty |\langle f,\varphi_n(\omega)\rangle|^2.~~a.s.
\]
\end{theorem}

\section{Asymptotic results for the eigenvalues}
In this section, we derive the stochastic Pr\"ufer equations and prove the high-energy expansion of the Pr\"ufer angle, which is useful for studying asymptotic results of eigenvalues under separated boundary value conditions. As a typical case, we give the Dirichlet eigenvalue asymptotics with an explicit stochastic correction and the almost-sure estimate.

\subsection{Stochastic equation and Pr\"ufer variables}
Write $\lambda=k^2>0$. The eigenvalue equation $\ell_{\sigma_\omega}y=k^2y$ is equivalent to the first-order system
\[
    \begin{cases}
    Y'=Q+\sigma Y,\\[2mm]
    Q'=-(k^2+\sigma^2)Y-\sigma Q,
    \end{cases}
\]
where \( 
Y=y,
\quad
Q=y^{[1]}.\)
Equivalently, if one introduces the classical momentum variable
\[
    P:=Y'=Q+\sigma Y,
\]
then the same equation becomes the SDE system
\begin{equation}\label{SDE}
\left\{
    \begin{aligned}
    &\dd Y_x=P_x\,\dd x,\\[2mm]
    &\dd P_x=-k^2Y_x\,\dd x+\rho Y_x\,\dd B_x.
    \end{aligned}\right.
\end{equation}
Actually, since $Q=P-\sigma Y$ and $\dd \sigma=\rho\,\dd B$, we have
\[
    \dd Q=\dd P-\dd (\sigma Y)
       =\dd P-\sigma\,\dd Y-Y\,\dd \sigma
\]
by It\^o's formula. Therefore 
\[
    Q'=-k^2Y-\sigma P
\]
is equivalent to
\[
    \dd P=-k^2Y\,\dd x+\rho Y\,\dd B_x.
\]
Any solution of the SDE \eqref{SDE} gives a solution of the quasi-derivative equation by setting $Q=P-\sigma Y$, therefore, \(k^2\) is an eigenvalue if and only if the solution of \eqref{SDE} satisfies the corresponding boundary conditions. This equivalence is understood as follows. For each fixed Brownian path outside a null set, the quasi-derivative problem is a deterministic Sturm--Liouville problem with primitive $\sigma=\rho B$. The classical momentum $P=Y'$ is then a semimartingale in the probabilistic representation, and the identity $Q=P-\sigma Y$ converts the pathwise quasi-derivative system into the It\^o system \eqref{SDE}. Conversely, a solution of \eqref{SDE} with $Q=P-\sigma Y$ satisfies the quasi-derivative equation in the distributional sense and preserves the boundary conditions. We shall use this stochastic first-order system to investigate the eigenvalues.

For a non-trivial solution, define the lifted Pr\"ufer variables by
\[
    Y_x=R_x\sin\theta_x,
    \qquad
    \frac{P_x}{k}=R_x\cos\theta_x,
    \qquad R_x>0,
\]
where $\theta$ is chosen continuously. The use of a lifted angle avoids the branch ambiguity of the arctangent.

\begin{lemma}[Pr\"ufer angle and amplitude equations]\label{lem:prufer}
The lifted angle satisfies
\begin{equation}\label{eq:prufer-angle}
    \dd\theta_x
    =k\,\dd x-\frac{\rho}{k}\sin^2\theta_x\,\dd B_x
    +\frac{\rho^2}{k^2}\sin^3\theta_x\cos\theta_x\,\dd x.
\end{equation}
Moreover,
\begin{equation}\label{eq:logR}
    \dd\log R_x
    =\frac{\rho}{2k}\sin(2\theta_x)\,\dd B_x
    -\frac{\rho^2}{2k^2}\sin^2\theta_x\cos(2\theta_x)\,\dd x.
\end{equation}
\end{lemma}

\begin{proof}
Apply It\^o's formula to $\Phi(Y,P)=\arctan(kY/P)$ locally on coordinate charts, and then patch the result for the lifted angle. The derivatives are
\[
    \Phi_Y=\frac{kP}{k^2Y^2+P^2},\qquad
    \Phi_P=-\frac{kY}{k^2Y^2+P^2},\qquad
    \Phi_{PP}=\frac{2kYP}{(k^2Y^2+P^2)^2}.
\]
Since the only quadratic variation comes from $\dd P$, It\^o's formula gives
\[
    \dd\theta
    =k\,\dd x-\frac{\rho kY^2}{k^2Y^2+P^2}\,\dd B
    +\frac{\rho^2kY^3P}{(k^2Y^2+P^2)^2}\,\dd x.
\]
Substituting $Y=R\sin\theta$ and $P=kR\cos\theta$ yields \eqref{eq:prufer-angle}.

For the amplitude, write $R^2=Y^2+(P/k)^2$. A direct It\^o calculation gives
\[
    \dd\log R
    =\frac{\rho YP}{k^2R^2}\,\dd B
    +\frac{\rho^2}{2k^2}\left(\frac{Y^2}{R^2}-\frac{2Y^2P^2}{k^2R^4}\right)\dd x.
\]
Again substituting $Y=R\sin\theta$, $P=kR\cos\theta$ gives \eqref{eq:logR}.
\end{proof}

\subsection{High-energy expansion of the Pr\"ufer angle}

Now at an endpoint $a\in\{0,L\}$ with separated parameter $\gamma\in[0,\pi)$, the quasi-derivative boundary condition
\[
    \cos\gamma\,Y(a)-\sin\gamma\,Q(a)=0
\]
becomes, in the $(Y,P)$-Pr\"ufer angle,
\[
    \bigl(\cos\gamma+\sigma(a)\sin\gamma\bigr)\sin\theta(a)
    -k\sin\gamma\cos\theta(a)=0.
\]
Thus, for general separated boundary conditions, the eigenvalue condition is obtained as follows. Choose an initial angle $\vartheta_{\alpha,0}(k)$ satisfying
\[
    \bigl(\cos\alpha+\sigma(0)\sin\alpha\bigr)\sin\vartheta_{\alpha,0}(k)
    -k\sin\alpha\cos\vartheta_{\alpha,0}(k)=0,
\]
solve (2) with $\theta_0=\vartheta_{\alpha,0}(k)$, and impose at $x=L$
\[
    \bigl(\cos\beta+\sigma(L)\sin\beta\bigr)\sin\theta_L(k)
    -k\sin\beta\cos\theta_L(k)=0.
\]
Equivalently, if $\vartheta_{\beta,L}(k)$ is a terminal boundary phase satisfying
\[
    \bigl(\cos\beta+\sigma(L)\sin\beta\bigr)\sin\vartheta_{\beta,L}(k)
    -k\sin\beta\cos\vartheta_{\beta,L}(k)=0,
\]
then the eigenvalue condition is
\[
    \theta_L(k)-\vartheta_{\beta,L}(k)\in\pi\mathbb Z.
\]
When $\sin\gamma=0$, the corresponding boundary phase is $0$ modulo $\pi$. For convenience,  in what follows, we specialize to the Dirichlet case. 

In the Dirichlet case we choose the solution with initial data $Y_0=0$ and $P_0=k$, so that $R_0=1$ and $\theta_0=0$. If the $n$-th Dirichlet eigenvalue $\lambda_n$ is positive and $k_n=\sqrt{\lambda_n}$, then the Dirichlet eigenvalue condition is
\[
    \theta_L(k_n)=n\pi .
\]
The convention $k_n=\sqrt{\lambda_n^+}$ will be used below in the $L^p$ statements; Lemma~\ref{lem:rough-moment} shows that this convention is harmless in the high-energy limit.

Set
\[
    h(u)=\sin^2u,
    \qquad h'(u)=\sin(2u),
    \qquad g(u)=\sin^3u\cos u.
\]
For $k>0$ define
\begin{align*}
    M_x(k)&=\int_0^x h(ks)\,\dd B_s,\\
    J_x(k)&=\int_0^x h'(ks)M_s(k)\,\dd B_s,\\
    A_x(k)&=\int_0^x g(ks)\,\dd s.
\end{align*}
The first two integrals are It\^o integrals for deterministic $k$, while $A_x(k)$ is a Lebesgue integral.

\begin{proposition}[Uniform second-order expansion]\label{prop:angle-expansion}
For every $p\in[2,\infty)$ there exists $C_p<\infty$, depending only on $p,L,$ and $\rho$, such that as $k\to\infty$,
\begin{equation}\label{eq:theta-expansion}
    \theta_x(k)
    =kx-\frac{\rho}{k}M_x(k)
    +\frac{\rho^2}{k^2}\bigl(J_x(k)+A_x(k)\bigr)
    +R_x(k),
\end{equation}
with
\[
    \left\|\sup_{0\le x\le L}|R_x(k)|\right\|_{L^p(\Omega)}\le C_pk^{-3}.
\]
Furthermore,
\[
    \left\|\sup_{x\le L}|M_x(k)|\right\|_{L^p(\Omega)}
    +\left\|\sup_{x\le L}|J_x(k)|\right\|_{L^p(\Omega)}
    \le C_p,
\]
and $\sup_{x\le L}|A_x(k)|\le (4k)^{-1}$. In particular, $\theta_x(k)=kx-\frac{\rho}{k}M_x(k)+O_{L^p}(k^{-2})$.
\end{proposition}

\begin{proof}
Let $\delta_x=\theta_x-kx$. The integrated form of \eqref{eq:prufer-angle} is
\begin{equation}\label{eq:delta-int}
    \delta_x=-\frac{\rho}{k}\int_0^xh(\theta_s)\,\dd B_s
    +\frac{\rho^2}{k^2}\int_0^xg(\theta_s)\,\dd s.
\end{equation}
Note that $\theta$ is continuous and adapted; $h$ is continuous, thus $h(\theta_s)$ is predictable.
Hence
\[
\int_0^x h(\theta_s)\,\dd B_s
\]
is a continuous square-integrable martingale with quadratic variation
\[
    \int_0^x |h(\theta_s)|^2\,\dd s.
\]. 

Recall the Burkholder--Davis--Gundy inequality. If $M_t=\int_0^t H_s\,\dd B_s$ is a continuous martingale, then
\[
  \left\|
\sup_{0\le t\le T}|M_t|
\right\|_{L^p(\Omega)}
\le
C_p
\left\|
\langle M\rangle_T^{1/2}
\right\|_{L^p(\Omega)},
\]
where $1 \le p < +\infty$ and $\langle M\rangle_T =\int _0^T H_s^2 ds$ denotes the predictable quadratic variation.
Since $h$ and $g$ are bounded, the Burkholder--Davis--Gundy (BDG) inequality gives
\[
    \left\|\sup_{x\le L}|\delta_x|\right\|_{L^p(\Omega)}\le C_pk^{-1}.
\]
Moreover, since $h$ is Lipschitz, another application of BDG gives
\[
\begin{aligned}
&\left\|\sup_{x\le L}
\left|\int_0^x\bigl(h(\theta_s)-h(ks)\bigr)\,\dd B_s\right|
\right\|_{L^p(\Omega)}  \\
&\qquad\le C_p
\left\|
\left(\int_0^L |\delta_s|^2\,\dd s\right)^{1/2}
\right\|_{L^p(\Omega)}
\le C_p\left\|\sup_{s\le L}|\delta_s|\right\|_{L^p(\Omega)}
\le C_pk^{-1}.
\end{aligned}
\]
Consequently,
\[
    \delta_x+\frac{\rho}{k}M_x(k)=O_{L^p}(k^{-2})
\]
uniformly in $x\in[0,L]$.

Moreover, Taylor's formula gives
\[
    h(\theta_s)=h(ks)+h'(ks)\delta_s+r_s,
    \qquad |r_s|\le C|\delta_s|^2.
\]
By the BDG inequality and the preceding bound, applied in the higher moment $2p$,
\[
\begin{aligned}
    \left\|\sup_{x\le L}\left|\int_0^xr_s\,\dd B_s\right|\right\|_{L^p(\Omega)}
    &\le C_p\left\|\left(\int_0^L|\delta_s|^4\,\dd s\right)^{1/2}\right\|_{L^p(\Omega)} \\
    &\le C_{p,L}\left\|\sup_{s\le L}|\delta_s|^2\right\|_{L^p(\Omega)}
     = C_{p,L}\left\|\sup_{s\le L}|\delta_s|\right\|_{L^{2p}(\Omega)}^2
    \le C_pk^{-2}.
\end{aligned}
\]
Consequently,
\[
    \int_0^xh'(ks)\delta_s\,\dd B_s
    =-\frac{\rho}{k}\int_0^xh'(ks)M_s(k)\,\dd B_s+O_{L^p}(k^{-2})
\]
uniformly in $x$. Thus
\[
    \int_0^xh(\theta_s)\,\dd B_s
    =M_x(k)-\frac{\rho}{k}J_x(k)+O_{L^p}(k^{-2}).
\]
Similarly, since $g$ is Lipschitz and $\delta=O_{L^p}(k^{-1})$ uniformly,
\[
    \int_0^xg(\theta_s)\,\dd s
    =A_x(k)+O_{L^p}(k^{-1}).
\]
Considering the above estimates, the equation \eqref{eq:delta-int} gives \eqref{eq:theta-expansion} with the claimed $O_{L^p}(k^{-3})$ remainder.

The martingale bounds for $M$ and $J$ follow from the BDG and boundedness of $h,h'$.
Finally,
\[
    A_x(k)=\int_0^x\sin^3(ks)\cos(ks)\,\dd s
    =\frac{\sin^4(kx)}{4k},
\]
so $\sup_x|A_x(k)|\le(4k)^{-1}$. 

Hence the proposition is proved.
\end{proof}

\begin{remark}[Range of $p$]\label{rem:p-range}
We state Proposition \ref{prop:angle-expansion} for $p\in[2,\infty)$ because this range is convenient when estimating quadratic Taylor remainders such as $|\theta_s-ks|^2$ by elementary supremum estimates and H\"older's inequality. The restriction is not essential. The BDG inequality is valid for every finite $p\ge1$, and the estimates for $p\in[1,2)$ follow from the monotonicity of $L^p$-norms on a probability space. The endpoint $p=\infty$ is excluded in the present moment estimates.
\end{remark}

\begin{remark}
The statement is an $L^p$ and in-probability expansion for deterministic values of $k$. Almost-sure estimates along the discrete eigenvalue sequence are obtained later by combining finite-moment bounds with the Borel--Cantelli lemma.
\end{remark}

\subsection{Dirichlet eigenvalue asymptotics}

Let $(\lambda_n)_{n\ge1}$ denote the Dirichlet eigenvalues of $H_{\omega}$, ordered increasingly. Throughout the rest of the paper we set
\[
    \lambda_n^+:=\max\{\lambda_n,0\},
    \qquad
    k_n:=\sqrt{\lambda_n^+}.
\]
Lemma~\ref{lem:rough-moment} below implies that, almost surely, $\lambda_n>0$ for all sufficiently large $n$, and that the exceptional event $\{\lambda_n\le0\}$ has faster-than-polynomial probability decay. Hence this convention does not affect any of the high-energy $L^p$ or almost-sure asymptotics. Before establishing the asymptotic results, it is appropriate to give some auxiliary parameter-uniform estimates for the Pr\"ufer phase and its derivatives first.

\begin{lemma}[Parameter-uniform phase estimates]\label{lem:phase-uniform}
Let
\[
    I_n=[K_n-1,K_n+1],\qquad K_n=\frac{n\pi}{L}.
\]
For every finite $p\ge2$ there is a constant $C_p$, independent of $n$, such that, for all sufficiently large $n$,
\begin{align*}
    \left\|\sup_{k\in I_n}\sup_{x\le L}|\theta_x(k)-kx|\right\|_{L^p(\Omega)}
    &\le C_p n^{-1},\\
    \left\|\sup_{k\in I_n}\sup_{x\le L}|\partial_k\theta_x(k)-x|\right\|_{L^p(\Omega)}
    &\le C_p n^{-1}.
\end{align*}
\end{lemma}

\begin{proof}
We write the proof in some detail, since this is the point where one has to distinguish fixed-frequency BDG estimates from estimates which are uniform in the parameter $k$.

First fix a deterministic $k\ge1$ and put
\[
    \delta_x(k):=\theta_x(k)-kx .
\]
The integrated equation \eqref{eq:delta-int} gives
\[
    \delta_x(k)
    =-\frac{\rho}{k}\int_0^x h(\theta_s(k))\,\dd B_s
      +\frac{\rho^2}{k^2}\int_0^x g(\theta_s(k))\,\dd s .
\]
Since $h$ and $g$ are bounded, BDG implies
\[
\begin{aligned}
    \left\|\sup_{x\le L}|\delta_x(k)|\right\|_{L^p(\Omega)}
    &\le \frac{C_p}{k}
       \left\|\left(\int_0^L |h(\theta_s(k))|^2\,\dd s\right)^{1/2}\right\|_{L^p(\Omega)}
       +\frac{C}{k^2}\int_0^L\dd s  \\
    &\le C_p k^{-1}.
\end{aligned}
\]
This gives the first fixed-frequency bound.

Next set
\[
    \zeta_x(k):=\partial_k\theta_x(k),\qquad e_x(k):=\zeta_x(k)-x .
\]
The differentiability with respect to $k$ follows from the usual smooth parameter-dependence theorem for strong solutions of SDEs. In the estimates below we work on the compact intervals $I_n=[K_n-1,K_n+1]$; for all large $n$, these intervals lie in a fixed half-line $[k_0,\infty)$ with $k_0>0$. On such sets the coefficients
\[
    b(k,u)=k+\frac{\rho^2}{k^2}g(u),
    \qquad
    a(k,u)=-\frac{\rho}{k}h(u)
\]
are smooth in $(k,u)$, the relevant $k$-derivatives are locally bounded, and all derivatives in $u$ are bounded uniformly. Hence, after choosing a modification if necessary, $(k,x)\mapsto\theta_x(k)$ is twice continuously differentiable in $k$ on compact $k$-intervals, uniformly for $x\in[0,L]$ almost surely, and the $k$-derivatives solve the equations obtained by differentiating the stochastic equation. Differentiating \eqref{eq:prufer-angle} gives
\[
\begin{aligned}
    \dd\zeta_x
    &=\dd x+\frac{\rho}{k^2}h(\theta_x)\,\dd B_x
      -\frac{\rho}{k}h'(\theta_x)\zeta_x\,\dd B_x \\
    &\quad -\frac{2\rho^2}{k^3}g(\theta_x)\,\dd x
      +\frac{\rho^2}{k^2}g'(\theta_x)\zeta_x\,\dd x,
    \qquad \zeta_0=0.
\end{aligned}
\]
Subtracting $x$ and using $\zeta_s=s+e_s$, we obtain
\[
\begin{aligned}
    e_x
    ={}&\frac{\rho}{k^2}\int_0^x h(\theta_s)\,\dd B_s
      -\frac{\rho}{k}\int_0^x h'(\theta_s)s\,\dd B_s
      -\frac{\rho}{k}\int_0^x h'(\theta_s)e_s\,\dd B_s \\
    &-\frac{2\rho^2}{k^3}\int_0^x g(\theta_s)\,\dd s
      +\frac{\rho^2}{k^2}\int_0^x g'(\theta_s)s\,\dd s
      +\frac{\rho^2}{k^2}\int_0^x g'(\theta_s)e_s\,\dd s .
\end{aligned}
\]
Let
\[
    A_p(k):=\left\|\sup_{x\le L}|e_x(k)|\right\|_{L^p(\Omega)} .
\]
For the deterministic drift terms, boundedness of $g'$ gives
\[
    \left\|\sup_{x\le L}\left|\frac{\rho^2}{k^2}\int_0^x g'(\theta_s)s\,\dd s\right|\right\|_{L^p(\Omega)}
    \le Ck^{-2},
\]
Similarly, the term with $g$ is $O(k^{-3})$. For the last drift term,
\[
\begin{aligned}
    \left\|\sup_{x\le L}\left|\frac{\rho^2}{k^2}\int_0^x g'(\theta_s)e_s\,\dd s\right|\right\|_{L^p(\Omega)}
    &\le Ck^{-2}\int_0^L \|e_s\|_{L^p(\Omega)}\,\dd s  \\
    &\le Ck^{-2}A_p(k).
\end{aligned}
\]
For the stochastic terms, BDG gives
\[
    \left\|\sup_{x\le L}\left|\frac{\rho}{k^2}\int_0^x h(\theta_s)\,\dd B_s\right|\right\|_{L^p(\Omega)}
    \le C_p k^{-2},
\]
\[
    \left\|\sup_{x\le L}\left|\frac{\rho}{k}\int_0^x h'(\theta_s)s\,\dd B_s\right|\right\|_{L^p(\Omega)}
    \le \frac{C_p}{k}\left(\int_0^L s^2\,\dd s\right)^{1/2}
    \le C_p k^{-1},
\]
and
\[
\begin{aligned}
    \left\|\sup_{x\le L}\left|\frac{\rho}{k}\int_0^x h'(\theta_s)e_s\,\dd B_s\right|\right\|_{L^p(\Omega)}
    &\le \frac{C_p}{k}
      \left\|\left(\int_0^L |e_s|^2\,\dd s\right)^{1/2}\right\|_{L^p(\Omega)}  \\
    &\le C_p k^{-1}A_p(k).
\end{aligned}
\]
Combining these estimates yields
\[
    A_p(k)\le C_p k^{-1}+C_p(k^{-1}+k^{-2})A_p(k).
\]
Choose $k_0$ so large that $C_p(k^{-1}+k^{-2})\le1/2$ for $k\ge k_0$. Then the last term can be absorbed into the left-hand side, and hence
\begin{equation}\label{eq:e-fixed-bound}
    \left\|\sup_{x\le L}|\partial_k\theta_x(k)-x|\right\|_{L^p(\Omega)}
    =A_p(k)
    \le C_p k^{-1}.
\end{equation}

We shall also need the same size bound for the $k$-derivative of $e$. Put
\[
    \eta_x(k):=\partial_k^2\theta_x(k)=\partial_k e_x(k).
\]
Differentiating the equation for $\zeta$ once more gives
\[
\begin{aligned}
\dd\eta_x={}&
\left[-\frac{2\rho}{k^3}h(\theta_x)
      +\frac{2\rho}{k^2}h'(\theta_x)\zeta_x
      -\frac{\rho}{k}h''(\theta_x)\zeta_x^2
      -\frac{\rho}{k}h'(\theta_x)\eta_x\right]\dd B_x \\
&+\left[\frac{6\rho^2}{k^4}g(\theta_x)
      -\frac{4\rho^2}{k^3}g'(\theta_x)\zeta_x
      +\frac{\rho^2}{k^2}g''(\theta_x)\zeta_x^2
      +\frac{\rho^2}{k^2}g'(\theta_x)\eta_x\right]\dd x,
\end{aligned}
\]
with $\eta_0=0$. Since $\zeta_x=x+e_x$, the preceding estimate gives
\[
    \left\|\sup_{x\le L}|\zeta_x(k)|\right\|_{L^p(\Omega)}\le C_p.
\]
Using also that $h,h',h'',g,g',g''$ are bounded, BDG and H\"older's inequality imply, with
\[
    B_p(k):=\left\|\sup_{x\le L}|\eta_x(k)|\right\|_{L^p(\Omega)},
\]
that
\[
\begin{aligned}
B_p(k)
&\le C_p\left(k^{-3}+k^{-2}\left\|\sup_{x\le L}|\zeta_x|\right\|_{L^p(\Omega)}
        +k^{-1}\left\|\sup_{x\le L}|\zeta_x|^2\right\|_{L^p(\Omega)}\right) \\
&\quad +C_p k^{-1}
   \left\|\left(\int_0^L |\eta_s|^2\,\dd s\right)^{1/2}\right\|_{L^p(\Omega)} \\
&\quad +C_p\left(k^{-4}+k^{-3}\left\|\sup_{x\le L}|\zeta_x|\right\|_{L^p(\Omega)}
        +k^{-2}\left\|\sup_{x\le L}|\zeta_x|^2\right\|_{L^p(\Omega)}\right)
      +C_p k^{-2}\int_0^L\|\eta_s\|_{L^p(\Omega)}\,\dd s  \\
&\le C_p k^{-1}+C_p(k^{-1}+k^{-2})B_p(k).
\end{aligned}
\]
Here $p\ge2$ is used only to control $\|\sup_x|\zeta_x|^2\|_{L^p(\Omega)}$ by the fixed-frequency bounds in a higher moment. Enlarging $k_0$ if necessary, we absorb the last term and obtain
\begin{equation}\label{eq:eta-fixed-bound}
    \left\|\sup_{x\le L}|\partial_k^2\theta_x(k)|\right\|_{L^p(\Omega)}
    \le C_p k^{-1},
    \qquad k\ge k_0 .
\end{equation}

It remains to pass from deterministic $k$ to the supremum over $I_n$. We use the Banach-valued one-dimensional Sobolev inequality. If $f:I_n\to C[0,L]$ is absolutely continuous and $p>1$, then
\[
    \sup_{k\in I_n}\|f(k)\|_{C[0,L]}
    \le C
    \left(
      \left(\int_{I_n}\|f(k)\|_{C[0,L]}^p\,\dd k\right)^{1/p}
      +
      \left(\int_{I_n}\|f'(k)\|_{C[0,L]}^p\,\dd k\right)^{1/p}
    \right),
\]
where $C$ depends only on the length $|I_n|=2$. Apply this with
\[
    f(k)=\theta_\cdot(k)-k(\cdot),
    \qquad f'(k)=\partial_k\theta_\cdot(k)-\cdot.
\]
Taking $L^p(\Omega)$-norms and using Fubini's theorem, together with the fixed-frequency bounds, we obtain
\[
\begin{aligned}
&\left\|\sup_{k\in I_n}\sup_{x\le L}|\theta_x(k)-kx|\right\|_{L^p(\Omega)} \\
&\qquad \le C_p
    \left(\int_{I_n} k^{-p}\,\dd k\right)^{1/p}
    +C_p
    \left(\int_{I_n} k^{-p}\,\dd k\right)^{1/p}
    \le C_p n^{-1}.
\end{aligned}
\]
Applying the same Sobolev inequality with
\[
    f(k)=\partial_k\theta_\cdot(k)-\cdot,
    \qquad f'(k)=\partial_k^2\theta_\cdot(k),
\]
and using \eqref{eq:e-fixed-bound} and \eqref{eq:eta-fixed-bound}, we obtain
\[
    \left\|\sup_{k\in I_n}\sup_{x\le L}|\partial_k\theta_x(k)-x|\right\|_{L^p(\Omega)}
    \le C_p n^{-1}.
\]
This proves the lemma.
\end{proof}

\begin{lemma}[Positivity and rough moment bounds for the high-energy roots]\label{lem:rough-moment}
For every finite \(q,r\ge1\) there are constants \(C_q,C_r\) such that, for all sufficiently large \(n\),
\[
    \Prob(\lambda_n\le0)\le C_q n^{-q},
    \qquad
    \|k_n\|_{L^r(\Omega)}\le C_r n,
    \qquad
    \|\lambda_n\|_{L^r(\Omega)}\le C_r n^2 .
\]
Consequently, almost surely, \(\lambda_n>0\) for all sufficiently large \(n\), and then \(k_n=\sqrt{\lambda_n}\) eventually.
\end{lemma}

\begin{proof}
 The quadratic form corresponding to the quasi-derivative realization may be written as
\[
    \mathfrak t_\omega[y]
    =\int_0^L |y'|^2\,\dd x
      -2\operatorname{Re}\int_0^L \sigma_\omega y'\overline y\,\dd x,
    \qquad y\in H_0^1(0,L).
\]
By the one-dimensional Gagliardo--Nirenberg inequality and Young's inequality,
\[
    \left|2\int_0^L\sigma_\omega y'\overline y\,\dd x\right|
    \le C_L\|\sigma_\omega\|_{L^2(0,L)}\|y'\|_{L^2(0,L)}^{3/2}\|y\|_{L^2(0,L)}^{1/2}
    \le \frac12\|y'\|_{L^2(0,L)}^2
        +C_L\|\sigma_\omega\|_{L^2(0,L)}^4\|y\|_{L^2(0,L)}^2.
\]
Therefore
\[
    \frac12\|y'\|_{L^2(0,L)}^2
      -C\|\sigma_\omega\|_{L^2(0,L)}^4\|y\|_{L^2(0,L)}^2
    \le
    \mathfrak t_\omega[y]
    \le
    \frac32\|y'\|_{L^2(0,L)}^2
      +C\|\sigma_\omega\|_{L^2(0,L)}^4\|y\|_{L^2(0,L)}^2 .
\]
Let \(\mu_n=(n\pi/L)^2\) be the Dirichlet eigenvalues of \(-\dd^2/\dd x^2\). The min--max principle gives
\[
    \lambda_n(\omega)
    \ge \frac12 \mu_n-C\|\sigma_\omega\|_{L^2(0,L)}^4
    \ge c n^2-C\|\sigma_\omega\|_{L^2(0,L)}^4 .
\]
Similarly, by taking in the min--max principle the span of the first \(n\) Dirichlet sine functions,
\[
    \lambda_n(\omega)
    \le C n^2+C\|\sigma_\omega\|_{L^2(0,L)}^4 .
\]
The lower bound implies
\[
    \{\lambda_n\le0\}
    \subset
    \{\|\sigma_\omega\|_{L^2(0,L)}^4\ge c n^2\}.
\]
We now spell out the probabilistic estimate. Since
\[
    \|\sigma_\omega\|_{L^2(0,L)}=|\rho|\,\|B(\omega)\|_{L^2(0,L)},
\]
it is enough to use high moments of the Brownian \(L^2(0,L)\)-norm. For every integer \(m\ge1\), Jensen's inequality gives
\[
\begin{aligned}
    \E\|B\|_{L^2(0,L)}^{4m}
    &=\E\left(\int_0^L B_x^2\,\dd x\right)^{2m}  \\
    &\le L^{2m-1}\int_0^L \E |B_x|^{4m}\,\dd x
      = C_m\int_0^L x^{2m}\,\dd x <\infty,
\end{aligned}
\]
where we used that \(B_x\) is a centered Gaussian random variable with variance \(x\). Hence
\[
    \E\|\sigma_\omega\|_{L^2(0,L)}^{4m}<\infty
    \qquad\text{for every }m\ge1.
\]
By Markov's inequality, for every integer \(m\ge1\),
\[
\begin{aligned}
    \Prob(\lambda_n\le0)
    &\le \Prob\bigl(\|\sigma_\omega\|_{L^2(0,L)}^4\ge c n^2\bigr)  \\
    &=\Prob\bigl(\|\sigma_\omega\|_{L^2(0,L)}^{4m}\ge c^m n^{2m}\bigr)  \\
    &\le c^{-m}n^{-2m}\,
       \E\|\sigma_\omega\|_{L^2(0,L)}^{4m}
      \le C_m n^{-2m}.
\end{aligned}
\]
Given any \(q<\infty\), choose \(m\) so large that \(2m\ge q\). Then, after changing the constant,
\[
    \Prob(\lambda_n\le0)\le C_q n^{-q}.
\]
The Borel--Cantelli lemma then gives \(\lambda_n>0\) eventually almost surely.

Finally, the upper and lower bounds imply
\[
    |\lambda_n(\omega)|\le Cn^2+C\|\sigma_\omega\|_{L^2(0,L)}^4 ,
\]
and the upper bound together with the definition \(k_n=\sqrt{\lambda_n^+}\) implies
\[
    k_n(\omega)\le Cn+C\|\sigma_\omega\|_{L^2(0,L)}^2 .
\]
Taking \(L^r(\Omega)\)-norms and using again the finite moments of \(\|B\|_{L^2(0,L)}\) gives
\[
    \|k_n\|_{L^r(\Omega)}\le C_r n,
    \qquad
    \|\lambda_n\|_{L^r(\Omega)}\le C_r n^2 .
\]
\end{proof}

\begin{theorem}[Dirichlet eigenvalue expansion]\label{thm:eigenvalue-asymptotics}
Let
\[
    K_n:=\frac{n\pi}{L},
    \qquad
    M_n:=M_L(K_n)
    =
    \int_0^L
    \sin^2(K_ns)\,\dd B_s .
\]
Then, for every finite \(p\ge1\),
\[
    k_n
    =
    K_n+\frac{\rho}{n\pi}M_n+O_{L^p}(n^{-2}),
    \qquad n\to\infty.
\]
Equivalently,
\[
    k_n
    =
    \frac{n\pi}{L}
    +
    \frac{\rho}{n\pi}
    \int_0^L
    \sin^2\left(\frac{n\pi s}{L}\right)\,\dd B_s
    +
    O_{L^p}(n^{-2}).
\]
Moreover,
\[
    \lambda_n
    =
    K_n^2+\frac{2\rho}{L}M_n+O_{L^p}(n^{-1}).
\]
\end{theorem}

\begin{proof}
It is enough to prove the estimate for $p\ge2$; the case $1\le p<2$ follows by monotonicity of $L^p$-norms. Put
\[
    I_n=[K_n-1,K_n+1]
\]
and define
\[
    E_n:=\sup_{k\in I_n}|\theta_L(k)-kL|,
    \qquad
    D_n:=\sup_{k\in I_n}|\partial_k\theta_L(k)-L|.
\]
By Lemma~\ref{lem:phase-uniform}, for every finite $q\ge2$,
\[
    \|E_n\|_{L^q(\Omega)}+\|D_n\|_{L^q(\Omega)}\le C_q n^{-1}.
\]
Let
\[
    G_n:=\{E_n<L/2\}\cap\{D_n<L/2\}.
\]
Then Markov's inequality gives, for every finite $q$,
\[
    \Prob(G_n^c)\le C_q n^{-q}.
\]
On $G_n$ we have $\partial_k\theta_L(k)\ge L/2$ throughout $I_n$. Hence $k\mapsto\theta_L(k)$ is strictly increasing on $I_n$. Moreover,
\[
    \theta_L(K_n+1)\ge n\pi+L-E_n>n\pi,
    \qquad
    \theta_L(K_n-1)\le n\pi-L+E_n<n\pi.
\]
Therefore there is a unique solution of $\theta_L(k)=n\pi$ in $I_n$. By the Sturm oscillation theorem for the Dirichlet problem with distributional potentials (in the quasi-derivative formulation, see for example \cite{SavchukShkalikov,HrynivMykytyuk2012}), this solution is the Dirichlet root $k_n$ corresponding to the phase level $n\pi$. In addition, if $r>E_n/L$, then
\[
    \theta_L(K_n+r)>n\pi,
    \qquad
    \theta_L(K_n-r)<n\pi,
\]
so by monotonicity
\[
    |k_n-K_n|\le \frac{E_n}{L}
    \qquad\text{on }G_n.
\]
Consequently,
\[
    \bigl\||k_n-K_n|\mathbf 1_{G_n}\bigr\|_{L^p(\Omega)}\le C_p n^{-1}.
\]
On the complement, Lemma~\ref{lem:rough-moment} and H\"older's inequality give, after choosing $q$ sufficiently large,
\[
\begin{aligned}
    \bigl\||k_n-K_n|\mathbf 1_{G_n^c}\bigr\|_{L^p(\Omega)}
    &\le \|k_n-K_n\|_{L^{2p}(\Omega)}\Prob(G_n^c)^{1/(2p)}  \\
    &\le C_p n^{-1}.
\end{aligned}
\]
Thus
\[
    k_n-K_n=O_{L^p}(n^{-1}).
\]

We now refine the estimate. Since $K_nL=n\pi$, Proposition~\ref{prop:angle-expansion} at the deterministic frequency $K_n$ gives
\[
    \theta_L(K_n)-n\pi
    =-
    \frac{\rho}{K_n}M_L(K_n)+O_{L^p}(n^{-2}).
\]
On $G_n$, the mean value formula gives
\[
\begin{aligned}
0
&=\theta_L(k_n)-n\pi \\
&=\theta_L(K_n)-n\pi
  +L(k_n-K_n)
  +\int_{K_n}^{k_n}\bigl(\partial_k\theta_L(u)-L\bigr)\,\dd u .
\end{aligned}
\]
The last term is $O_{L^p}(n^{-2})$ on $G_n$, because
\[
\begin{aligned}
&\left\|
    \mathbf 1_{G_n}
    \int_{K_n}^{k_n}\bigl(\partial_k\theta_L(u)-L\bigr)\,\dd u
\right\|_{L^p(\Omega)}  \\
&\qquad\le
\left\|\sup_{u\in I_n}|\partial_k\theta_L(u)-L|\right\|_{L^{2p}(\Omega)}
\bigl\||k_n-K_n|\mathbf 1_{G_n}\bigr\|_{L^{2p}(\Omega)}
\le C_p n^{-2}.
\end{aligned}
\]
Therefore, on $G_n$ in $L^p$,
\[
    L(k_n-K_n)=\frac{\rho}{K_n}M_L(K_n)+O_{L^p}(n^{-2}).
\]
It remains only to remove the indicator of $G_n$. The variable
\[
    k_n-K_n-\frac{\rho}{LK_n}M_L(K_n)
\]
has $L^r$-norm bounded by $C_r n$ for every finite $r$, by Lemma~\ref{lem:rough-moment} and the BDG estimate for $M_L(K_n)$. Since $\Prob(G_n^c)\le C_q n^{-q}$ for arbitrary finite $q$, H\"older's inequality makes the contribution of $G_n^c$ equal to $O_{L^p}(n^{-2})$. Hence
\[
    L(k_n-K_n)=\frac{\rho}{K_n}M_L(K_n)+O_{L^p}(n^{-2}),
\]
and since $LK_n=n\pi$,
\[
    k_n-K_n
    =\frac{\rho}{LK_n}M_L(K_n)+O_{L^p}(n^{-2})
    =\frac{\rho}{n\pi}M_L(K_n)+O_{L^p}(n^{-2}).
\]
It remains to derive the displayed expansion for \(\lambda_n\). On the event \(\{\lambda_n>0\}\), we have \(\lambda_n=k_n^2\), and hence
\[
\begin{aligned}
    \lambda_n-K_n^2
    &=2K_n(k_n-K_n)+(k_n-K_n)^2  \\
    &=\frac{2\rho K_n}{n\pi}M_L(K_n)+O_{L^p}(n^{-1})
     =\frac{2\rho}{L}M_L(K_n)+O_{L^p}(n^{-1}).
\end{aligned}
\]
The contribution of the event \(\{\lambda_n\le0\}\) is also \(O_{L^p}(n^{-1})\). Indeed, by Lemma~\ref{lem:rough-moment}, the variable
\[
    \lambda_n-K_n^2-\frac{2\rho}{L}M_L(K_n)
\]
has \(L^r\)-norm bounded by \(C_rn^2\), while \(\Prob(\lambda_n\le0)\le C_qn^{-q}\) for arbitrary finite \(q\); H\"older's inequality with \(q\) sufficiently large gives the claim.
This proves the theorem.
\end{proof}

\begin{corollary}[Almost-sure high-energy bound]
\label{cor:as-bound}
Let
\[
    K_n=\frac{n\pi}{L}.
\]
Then, for every \(\varepsilon>0\),
\[
    k_n=K_n+O(n^{-1+\varepsilon})
    \qquad\text{almost surely}.
\]
Consequently,
\[
    \lambda_n=K_n^2+O(n^{\varepsilon})
    \qquad\text{almost surely}
\]
for every \(\varepsilon>0\).
\end{corollary}

\begin{proof}
By Theorem \ref{thm:eigenvalue-asymptotics}, for every finite \(p\ge1\),
\[
    k_n=K_n+\frac{\rho}{n\pi}M_n+R_n,
    \qquad
    \|M_n\|_{L^p(\Omega)}\le C_p,
    \qquad
    \|R_n\|_{L^p(\Omega)}\le C_p n^{-2},
\]
where
\[
    M_n=\int_0^L\sin^2(K_ns)\,\dd B_s .
\]
Here the bound on \(M_n\) follows directly from the BDG inequality, since the integrand is deterministic and bounded.

Fix \(\varepsilon>0\). Choose \(p>1/\varepsilon\). By Markov's inequality,
\[
\begin{aligned}
    \Prob\left(\left|\frac{\rho}{n\pi}M_n\right|>n^{-1+\varepsilon}\right)
    &\le C_p n^{-p\varepsilon},\\
    \Prob\left(|R_n|>n^{-1+\varepsilon}\right)
    &\le C_p n^{-p(1+\varepsilon)}.
\end{aligned}
\]
Both right-hand sides are summable in \(n\). The Borel--Cantelli lemma therefore gives
\[
    \frac{\rho}{n\pi}M_n=O(n^{-1+\varepsilon}),
    \qquad
    R_n=O(n^{-1+\varepsilon})
    \qquad\text{almost surely}.
\]
This proves
\[
    k_n=K_n+O(n^{-1+\varepsilon})
    \qquad\text{almost surely}.
\]
Finally, by Lemma~\ref{lem:rough-moment}, almost surely \(\lambda_n>0\) for all sufficiently large \(n\). For those \(n\), \(\lambda_n=k_n^2\), and hence
\[
    \lambda_n-K_n^2=(k_n-K_n)(k_n+K_n).
\]
Since \(K_n\sim n\) and \(k_n-K_n=O(n^{-1+\varepsilon})\) almost surely, we obtain
\[
    \lambda_n=K_n^2+O(n^{\varepsilon})
    \qquad\text{almost surely}.
\]
\end{proof}

\section{Asymptotic results for eigenfunctions}

For $k>0$, let $(Y_x(k),P_x(k))_{0\le x\le L}$ be the solution of the initial value problem
\[
    \dd Y_x=P_x\,\dd x,
    \qquad
    \dd P_x=-k^2Y_x\,\dd x+\rho Y_x\,\dd B_x,
    \qquad
    Y_0=0,\quad P_0=k.
\]
We write
\[
    u_k(x):=Y_x(k),
    \qquad u_k'(x)=P_x(k).
\]
Equivalently, $u_k$ is the Dirichlet initial solution of
\[
    -u_k''+\rho\dot B\,u_k=k^2u_k
\]
in the quasi-derivative sense, normalized by $u_k(0)=0$ and $u_k'(0)=k$. The variation-of-constants formula gives the stochastic Volterra equation
\begin{equation}\label{eq:volterra}
    u_k(x)=\sin(kx)+\frac{\rho}{k}\int_0^x\sin(k(x-s))u_k(s)\,\dd B_s.
\end{equation}

\begin{proposition}[First-order eigenfunction expansion]\label{prop:eigenfunction}
For every $p\in[2,\infty)$ there exist constants $C_p<\infty$ and $k_0\ge1$ such that, for all $k\ge k_0$,
\begin{equation}\label{eq:uk-expansion}
    \left\|\sup_{0\le x\le L}\left|u_k(x)-\sin(kx)-\frac{\rho}{k}\int_0^x\sin(k(x-s))\sin(ks)\,\dd B_s\right|\right\|_{L^p(\Omega)}
    \le C_p k^{-2}.
\end{equation}
For random frequencies such as the eigenvalues $k_n$, the deterministic-frequency version of this expansion is stated below in Corollary~\ref{cor:as-eigenfunction}.
\end{proposition}

\begin{proof}
We first record the elementary decomposition
\[
    \sin(k(x-s))
    =
    \sin(kx)\cos(ks)-\cos(kx)\sin(ks).
\]
Thus the Volterra stochastic integral in \eqref{eq:volterra} can be written as
a linear combination, with bounded deterministic coefficients, of the two
martingales
\[
    C_x(k):=\int_0^x\cos(ks)u_k(s)\,\dd B_s,
    \qquad
    S_x(k):=\int_0^x\sin(ks)u_k(s)\,\dd B_s.
\]
The BDG inequality applied to these two martingales gives
\[
    \left\|\sup_{x\le L}|u_k(x)-\sin(kx)|\right\|_{L^p(\Omega)}
    \le
    \frac{C_p}{k}
    \left\|
        \left(\int_0^L |u_k(s)|^2\,\dd s\right)^{1/2}
    \right\|_{L^p(\Omega)}
    \le
    \frac{C_p}{k}
    \left\|\sup_{x\le L}|u_k(x)|\right\|_{L^p(\Omega)}.
\]
Since
\[
    \sup_{x\le L}|u_k(x)|
    \le 1+\sup_{x\le L}|u_k(x)-\sin(kx)|,
\]
the preceding estimate implies
\[
    \left\|\sup_{x\le L}|u_k(x)|\right\|_{L^p(\Omega)}
    \le 1+\frac{C_p}{k}
    \left\|\sup_{x\le L}|u_k(x)|\right\|_{L^p(\Omega)}.
\]
For $k\ge 2C_p$ the last term is absorbed into the left-hand side. Hence
\[
    \left\|\sup_{x\le L}|u_k(x)|\right\|_{L^p(\Omega)}\le C_p.
\]
Therefore
\[
    \left\|\sup_{x\le L}|u_k(x)-\sin(kx)|\right\|_{L^p(\Omega)}\le C_pk^{-1}.
\]
Substituting this estimate once more into \eqref{eq:volterra}, we obtain
\begin{align*}
&u_k(x)-\sin(kx)-\frac{\rho}{k}\int_0^x\sin(k(x-s))\sin(ks)\,\dd B_s\\
&\qquad=\frac{\rho}{k}\int_0^x\sin(k(x-s))\bigl(u_k(s)-\sin(ks)\bigr)\,\dd B_s.
\end{align*}
Using the same martingale decomposition of the kernel and BDG again yields the
\(O_{L^p}(k^{-2})\) bound uniformly in \(x\).
\end{proof}

\begin{lemma}[Deterministic-frequency and local derivative estimates]\label{lem:eigenfunction-uniform}
Let
\[
    K_n=\frac{n\pi}{L},\qquad \mathcal K_n=[K_n-1,K_n+1],
\]
and set
\[
    I_n(x):=\int_0^x\sin\left(K_n(x-s)\right)\sin(K_ns)\,\dd B_s .
\]
For every finite $p\ge2$ there exists $C_p$ such that, for all sufficiently large $n$,
\begin{align}
    \left\|\sup_{x\le L}
    \left|u_{K_n}(x)-\sin(K_nx)-\frac{\rho}{K_n}I_n(x)\right|\right\|_{L^p(\Omega)}
    &\le C_p n^{-2}, \label{eq:det-rem-un}\\
    \left\|\sup_{x\le L}|I_n(x)|\right\|_{L^p(\Omega)}
    &\le C_p, \label{eq:det-In-un}\\
    \E\left[\sup_{k\in\mathcal K_n}\sup_{x\le L}
    |\partial_k u_k(x)-x\cos(kx)|^p\right]
    &\le C_p n^{-p}. \label{eq:local-uk-der-un}
\end{align}
\end{lemma}

\begin{proof}
The first estimate is just Proposition~\ref{prop:eigenfunction} applied at the deterministic frequency $k=K_n$. The second estimate follows from BDG, because the integrand in $I_n$ is deterministic and uniformly bounded.

It remains to prove \eqref{eq:local-uk-der-un}. We first note that the pair $(u_k,u_k')$ solves the two-dimensional linear SDE
\[
    \dd Y_x=P_x\,\dd x,\qquad
    \dd P_x=-k^2Y_x\,\dd x+\rho Y_x\,\dd B_x,
    \qquad Y_0=0,\quad P_0=k.
\]
On each compact interval of parameters $k\in[k_0,K_0]$, $k_0>0$, the coefficients and the initial data are smooth in $k$ and globally Lipschitz in $(Y,P)$. Thus we may choose a version for which $u_k(x)$ is twice continuously differentiable in $k$, uniformly for $x\in[0,L]$ almost surely; the derivatives are obtained by differentiating the Volterra equation below. First fix a deterministic $k\ge k_0$. Differentiating the Volterra equation \eqref{eq:volterra}, and writing
\[
    v_k(x):=\partial_k u_k(x),
\]
we get
\[
\begin{aligned}
v_k(x)
={}&x\cos(kx)
-\frac{\rho}{k^2}\int_0^x \sin(k(x-s))u_k(s)\,\dd B_s \\
&+\frac{\rho}{k}\int_0^x (x-s)\cos(k(x-s))u_k(s)\,\dd B_s
+\frac{\rho}{k}\int_0^x \sin(k(x-s))v_k(s)\,\dd B_s .
\end{aligned}
\]
Set
\[
    w_k(x):=v_k(x)-x\cos(kx).
\]
Using the decompositions of $\sin(k(x-s))$ and $\cos(k(x-s))$ into products of functions of $x$ and functions of $s$, BDG and the bound
\[
    \left\|\sup_{x\le L}|u_k(x)|\right\|_{L^p(\Omega)}\le C_p
\]
from Proposition~\ref{prop:eigenfunction} give
\[
\begin{aligned}
\left\|\sup_{x\le L}|w_k(x)|\right\|_{L^p(\Omega)}
&\le C_pk^{-1}+C_pk^{-1}
\left\|\sup_{x\le L}|w_k(x)|\right\|_{L^p(\Omega)} .
\end{aligned}
\]
After absorbing the last term, for all sufficiently large $k$,
\begin{equation}\label{eq:w-fixed-bound}
    \left\|\sup_{x\le L}|\partial_k u_k(x)-x\cos(kx)|\right\|_{L^p(\Omega)}
    \le C_pk^{-1}.
\end{equation}
To pass from a fixed $k$ to the supremum over $\mathcal K_n$, we also need the same type of estimate for the $k$-derivative of $w_k$. For clarity, we write this equation explicitly. Put
\[
    S_k(x,s):=\sin(k(x-s)),\qquad C_k(x,s):=\cos(k(x-s)),
\]
and
\[
    z_k(x):=\partial_k w_k(x).
\]
Since $v_k(s)=s\cos(ks)+w_k(s)$, differentiating the equation for $w_k$ gives

\[
\begin{aligned}
 z_k(x)={}&\frac{2\rho}{k^3}\int_0^x S_k(x,s)u_k(s)\,\dd B_s \\
&-\frac{\rho}{k^2}\int_0^x (x-s)C_k(x,s)u_k(s)\,\dd B_s
 -\frac{\rho}{k^2}\int_0^x S_k(x,s)v_k(s)\,\dd B_s \\
&-\frac{\rho}{k^2}\int_0^x (x-s)C_k(x,s)u_k(s)\,\dd B_s \\
&-\frac{\rho}{k}\int_0^x (x-s)^2S_k(x,s)u_k(s)\,\dd B_s
 +\frac{\rho}{k}\int_0^x (x-s)C_k(x,s)v_k(s)\,\dd B_s \\
&-\frac{\rho}{k^2}\int_0^x S_k(x,s)s\cos(ks)\,\dd B_s \\
&+\frac{\rho}{k}\int_0^x (x-s)C_k(x,s)s\cos(ks)\,\dd B_s
 -\frac{\rho}{k}\int_0^x S_k(x,s)s^2\sin(ks)\,\dd B_s \\
&-\frac{\rho}{k^2}\int_0^x S_k(x,s)w_k(s)\,\dd B_s
 +\frac{\rho}{k}\int_0^x (x-s)C_k(x,s)w_k(s)\,\dd B_s \\
&+\frac{\rho}{k}\int_0^x S_k(x,s)z_k(s)\,\dd B_s .
\end{aligned}
\]
All kernels appearing here are uniformly bounded on $0\le s\le x\le L$, and the fixed-frequency bounds for $u_k$ and $w_k=\partial_k u_k-x\cos(kx)$ also give $\|\sup_{x\le L}|v_k(x)|\|_{L^p(\Omega)}\le C_p$. Hence BDG gives
\[
    \left\|\sup_{x\le L}|z_k(x)|\right\|_{L^p(\Omega)}
    \le C_pk^{-1}+C_pk^{-1}\left\|\sup_{x\le L}|z_k(x)|\right\|_{L^p(\Omega)}.
\]
Absorbing the last term for large $k$ yields
\begin{equation}\label{eq:dw-fixed-bound}
    \left\|\sup_{x\le L}|\partial_k w_k(x)|\right\|_{L^p(\Omega)}
    \le C_pk^{-1}.
\end{equation}
Equivalently,
\[
    \left\|\sup_{x\le L}\left|\partial_k\bigl(\partial_k u_k(x)-x\cos(kx)\bigr)\right|\right\|_{L^p(\Omega)}
    \le C_pk^{-1}.
\]
Now apply the one-dimensional Sobolev inequality in the parameter $k$ on the interval $\mathcal K_n$ to the Banach-valued function
\[
    f(k)=\partial_k u_k(\cdot)-(\cdot)\cos(k\cdot)\in C[0,L].
\]
Taking $L^p(\Omega)$-norms and using Tonelli's theorem, together with \eqref{eq:w-fixed-bound} and \eqref{eq:dw-fixed-bound}, gives
\[
\begin{aligned}
&\left\|\sup_{k\in\mathcal K_n}\sup_{x\le L}
|\partial_k u_k(x)-x\cos(kx)|\right\|_{L^p(\Omega)} \\
&\qquad\le C_p\left(\int_{\mathcal K_n} k^{-p}\,\dd k\right)^{1/p}
     +C_p\left(\int_{\mathcal K_n} k^{-p}\,\dd k\right)^{1/p}
\le C_p n^{-1}.
\end{aligned}
\]
This proves \eqref{eq:local-uk-der-un}.
\end{proof}

\begin{corollary}[Almost-sure eigenfunction asymptotics]
\label{cor:as-eigenfunction}
Let
\[
    K_n:=\frac{n\pi}{L},
    \qquad
    I_n(x):=\int_0^x\sin\left(K_n(x-s)\right)\sin(K_ns)\,\dd B_s.
\]
Here the stochastic integral is taken at the deterministic frequency $K_n$.
Then, for every $\varepsilon>0$,
\[
    \sup_{0\le x\le L}\left|u_{k_n}(x)-\sin(k_nx)-\frac{\rho}{k_n}I_n(x)\right|
    =
    O(n^{-2+\varepsilon})
    \qquad\text{a.s.}
\]
Moreover,
\[
    \sup_{0\le x\le L}|u_{k_n}(x)-\sin(k_nx)|
    =
    O(n^{-1+\varepsilon})
    \qquad\text{a.s.}
\]
\end{corollary}

\begin{proof}
Put
\[
    \Delta_n:=k_n-K_n.
\]
By Corollary~\ref{cor:as-bound},
\[
    \Delta_n=O(n^{-1+\varepsilon})
    \qquad\text{almost surely}
\]
for every $\varepsilon>0$. Hence, almost surely, $k_n\in\mathcal K_n$ for all sufficiently large $n$.

By \eqref{eq:det-rem-un}, \eqref{eq:det-In-un}, \eqref{eq:local-uk-der-un}, Markov's inequality and the Borel--Cantelli lemma, for every $\varepsilon>0$,
\begin{align*}
    \sup_{x\le L}\left|u_{K_n}(x)-\sin(K_nx)-\frac{\rho}{K_n}I_n(x)\right|
    &=O(n^{-2+\varepsilon}), \\
    \sup_{x\le L}|I_n(x)|
    &=O(n^{\varepsilon}), \\
    \sup_{k\in\mathcal K_n}\sup_{x\le L}|\partial_k u_k(x)-x\cos(kx)|
    &=O(n^{-1+\varepsilon})
\end{align*}
almost surely.
For such $n$, by the fundamental theorem of calculus in the parameter $k$,
\[
\begin{aligned}
&u_{k_n}(x)-\sin(k_nx)-\frac{\rho}{k_n}I_n(x) \\
&\quad=
\left(u_{K_n}(x)-\sin(K_nx)-\frac{\rho}{K_n}I_n(x)\right) \\
&\qquad+
\int_{K_n}^{k_n}\left(\partial_k u_k(x)-x\cos(kx)\right)\,\dd k
+\rho\left(\frac1{K_n}-\frac1{k_n}\right)I_n(x).
\end{aligned}
\]
Taking the supremum over $x$ gives
\[
\begin{aligned}
&\sup_{x\le L}\left|u_{k_n}(x)-\sin(k_nx)-\frac{\rho}{k_n}I_n(x)\right| \\
&\quad\le O(n^{-2+\varepsilon/3})
+O(n^{-1+\varepsilon/3})O(n^{-1+\varepsilon/3})
+C\,O(n^{-1+\varepsilon/3})n^{-2}O(n^{\varepsilon/3})
=O(n^{-2+\varepsilon}).
\end{aligned}
\]
Here we have applied the preceding almost-sure estimates with $\varepsilon/3$ in place of $\varepsilon$. This proves the first assertion.

The second assertion follows immediately from the first one and \eqref{eq:det-In-un}. Indeed,
\[
    \frac1{k_n}\sup_{x\le L}|I_n(x)|=O(n^{-1+\varepsilon})
    \qquad\text{almost surely},
\]
and therefore
\[
    \sup_{0\le x\le L}|u_{k_n}(x)-\sin(k_nx)|
    =O(n^{-1+\varepsilon})
    \qquad\text{almost surely}.
\]
\end{proof}

\begin{corollary}[Normalized Dirichlet eigenfunction expansion]\label{cor:normalized}
Let \(\varphi_n\) be the Dirichlet eigenfunction associated with \(\lambda_n\), normalized in \(L^2(0,L)\). We choose its sign so that \(\varphi_n'(0)>0\). Put
\[
    s_n(x):=\sin(k_nx),
    \qquad
    A_n:=\int_0^L s_n(t)I_n(t)\,\dd t,
\]
where \(I_n\) is the deterministic-frequency stochastic integral from Corollary~\ref{cor:as-eigenfunction}. Then, for every \(\varepsilon>0\),
\[
\begin{aligned}
&\sup_{0\le x\le L}\bigg|
\varphi_n(x)-\sqrt{\frac{2}{L}}
\left[
    s_n(x)+\frac{\rho}{k_n}\left(I_n(x)-\frac{2}{L}A_n s_n(x)\right)
\right]
\bigg|  \\
&\qquad =O(n^{-2+\varepsilon})
\qquad\text{almost surely}.
\end{aligned}
\]
Equivalently, the normalization constant of the unnormalized eigenfunction \(u_{k_n}\) satisfies
\[
    \|u_{k_n}\|_{L^2(0,L)}^{-1}
    =\sqrt{\frac{2}{L}}
    \left(1-\frac{2\rho}{Lk_n}A_n\right)
    +O(n^{-2+\varepsilon})
    \qquad\text{almost surely}.
\]
In particular,
\[
    \sup_{0\le x\le L}
    \left|\varphi_n(x)-\sqrt{\frac{2}{L}}\sin(k_nx)\right|
    =O(n^{-1+\varepsilon})
    \qquad\text{almost surely}.
\]
\end{corollary}

\begin{proof}
Almost surely, for all sufficiently large \(n\), \(\lambda_n>0\), and hence \(\varphi_n=u_{k_n}/\|u_{k_n}\|_{L^2(0,L)}\) with the prescribed sign. By Corollary~\ref{cor:as-eigenfunction},
\[
    u_{k_n}(x)=s_n(x)+\frac{\rho}{k_n}I_n(x)+R_n(x),
    \qquad
    \sup_{x\le L}|R_n(x)|=O(n^{-2+\varepsilon})
    \quad\text{a.s.}
\]
Moreover, by the BDG estimate \eqref{eq:det-In-un} and Borel--Cantelli,
\[
    \sup_{x\le L}|I_n(x)|=O(n^{\varepsilon})
    \qquad\text{a.s.}
\]
Consequently
\[
    |A_n|
    \le \|s_n\|_{L^1(0,L)}\sup_{x\le L}|I_n(x)|
    =O(n^{\varepsilon})
    \qquad\text{a.s.}
\]
We first expand the normalization factor. Since \(k_nL=n\pi+O(n^{-1+\varepsilon})\) almost surely,
\[
\begin{aligned}
    \|s_n\|_{L^2(0,L)}^2
    &=\int_0^L\sin^2(k_nx)\,\dd x
      =\frac{L}{2}-\frac{\sin(2k_nL)}{4k_n}
      =\frac{L}{2}+O(n^{-2+\varepsilon})
      \qquad\text{a.s.}
\end{aligned}
\]
Using the expansion of \(u_{k_n}\), we get
\[
\begin{aligned}
\|u_{k_n}\|_{L^2(0,L)}^2
&=\|s_n\|_{L^2(0,L)}^2
+\frac{2\rho}{k_n}\int_0^L s_n(t)I_n(t)\,\dd t
+O(n^{-2+\varepsilon})  \\
&=\frac{L}{2}+\frac{2\rho}{k_n}A_n+O(n^{-2+\varepsilon})
\qquad\text{a.s.}
\end{aligned}
\]
Here the terms involving \(R_n\) are \(O(n^{-2+\varepsilon})\). The quadratic term is also harmless: applying the bound \(\sup_x|I_n(x)|=O(n^{\varepsilon/3})\) gives
\[
    k_n^{-2}\|I_n\|_{L^2(0,L)}^2=O(n^{-2+2\varepsilon/3})=O(n^{-2+\varepsilon}).
\]
Since \(k_n^{-1}A_n=O(n^{-1+\varepsilon})\), Taylor expansion of \(x\mapsto x^{-1/2}\) at \(L/2\) gives
\[
    \|u_{k_n}\|_{L^2(0,L)}^{-1}
    =\sqrt{\frac{2}{L}}
    \left(1-\frac{2\rho}{Lk_n}A_n\right)
    +O(n^{-2+\varepsilon})
    \qquad\text{a.s.}
\]
Multiplying this expansion by
\[
    u_{k_n}=s_n+\frac{\rho}{k_n}I_n+R_n
\]
and using again \(\sup_x |I_n(x)|=O(n^{\varepsilon})\), \(|A_n|=O(n^{\varepsilon})\), and \(\sup_x|R_n(x)|=O(n^{-2+\varepsilon})\), we obtain
\[
\begin{aligned}
\varphi_n(x)
&=\sqrt{\frac{2}{L}}
\left(1-\frac{2\rho}{Lk_n}A_n\right)
\left(s_n(x)+\frac{\rho}{k_n}I_n(x)\right)
+O(n^{-2+\varepsilon}) \\
&=\sqrt{\frac{2}{L}}
\left[
    s_n(x)+\frac{\rho}{k_n}
    \left(I_n(x)-\frac{2}{L}A_n s_n(x)\right)
\right]
+O(n^{-2+\varepsilon})
\end{aligned}
\]
uniformly for \(0\le x\le L\). This proves the first-order normalized expansion. The final rough estimate follows because
\[
    \frac{1}{k_n}\sup_{x\le L}|I_n(x)|=O(n^{-1+\varepsilon}),
    \qquad
    \frac{|A_n|}{k_n}=O(n^{-1+\varepsilon})
    \qquad\text{a.s.}
\]
\end{proof}

\section{Discussion}

The finite-interval white-noise Schr\"odinger operator is well-defined pathwise through the quasi-derivative formalism. Its basic spectral properties follow from the regular theory of Sturm--Liouville operators with distributional potentials in \(H^{-1}(0,L)\). It has been shown that the stochastic Pr\"ufer equation gives a convenient high-energy description of the spectrum. Although the detailed asymptotic theorems in the present paper are proved for the Dirichlet problem, the method extends to general separated self-adjoint boundary conditions after adding the endpoint boundary phases described in Section~3.2. In such cases the leading Brownian correction still comes from the oscillatory integral in the interior, while the deterministic and random endpoint phases contribute additional lower-order boundary terms depending on the separated parameters and on $\sigma(0),\sigma(L)$. We have not pursued this bookkeeping here. However, coupled boundary conditions require a different analysis. Although their eigenvalues can often be localized between eigenvalues of nearby separated boundary value problems by comparison or interlacing arguments, such bounds are too coarse to recover the stochastic correction terms obtained here. In particular, the correction scale for \(k_n\) is \(O(n^{-1})\), or more softly \(O(n^{-1+\varepsilon})\) almost surely for every \(\varepsilon>0\), whereas interlacing with adjacent separated problems typically loses information at the level of the spacing between neighboring free frequencies. A sharp expansion for coupled conditions would require an analysis of the corresponding transfer matrix or characteristic determinant, and may involve paired eigenvalues as in the periodic and anti-periodic problems.

The almost-sure bounds proved here are deliberately obtained by a robust finite-moment and Borel--Cantelli argument. Since the leading terms are Gaussian stochastic integrals with deterministic oscillatory integrands, one may expect logarithmic refinements, or even law-of-the-iterated-logarithm type statements along subsequences, under a more detailed analysis of the correlations in $n$. Such refinements are not needed for the present high-energy expansion and are left for future work.

Moreover, we give some remarks on the stochastic characteristics of the asymptotic results.

\begin{remark}[Random structure of the high-energy spectrum and eigenfunctions]
The sequence $(k_n(\omega))_{n\ge1}$ may be viewed as a random point
configuration on $\mathbb R_+$,
\[
    \xi_\omega=\sum_{n\ge1}\delta_{k_n(\omega)}.
\]
It is not a Poisson point process.  Indeed, the almost-sure estimate
\[
    k_n=\frac{n\pi}{L}+O(n^{-1+\varepsilon})
    \qquad\text{a.s.}
\]
implies
\[
    k_{n+1}-k_n=\frac{\pi}{L}+O(n^{-1+\varepsilon})
    \qquad\text{a.s.}
\]
Hence the unfolded spacings satisfy
\[
    \frac{L}{\pi}(k_{n+1}-k_n)\to1
    \qquad\text{a.s.}
\]
This is a rigid, clock-type high-energy behaviour, in contrast with a
Poisson point process, whose spacings are exponentially distributed.

Nor should $(k_n)$ be regarded as a L\'evy process.  Even if the index $n$
is interpreted as a discrete time parameter, the increments
$k_{n+1}-k_n$ are neither stationary nor independent; all eigenvalues are
functions of the same Brownian path on $[0,L]$.

The eigenvalues $k_n$ themselves are random variables, not martingales,
because no natural filtration time is attached to the spectral index $n$.
Nevertheless, their leading fluctuations are expressed through the martingales
\[
    M_x(k)=\int_0^x\sin^2(ks)\,\dd B_s,
\]
defined first for deterministic values of the parameter $k$.  For each such
$k$, the process $x\mapsto M_x(k)$ is a continuous martingale with quadratic
variation
\[
    \langle M(k)\rangle_x=\int_0^x\sin^4(ks)\,\dd s.
\]
In the eigenvalue expansion,
\[
    n\left(k_n-\frac{n\pi}{L}\right)
    =
    \frac{\rho}{\pi}M_L(K_n)+O_{L^p}(n^{-1}),
    \qquad K_n=\frac{n\pi}{L},
\]
so the leading fluctuation is represented by deterministic-frequency Brownian
stochastic integrals.  This formulation avoids interpreting
$\int_0^L\sin^2(k_ns)\,\dd B_s$ as an anticipating It\^o integral.  The
fluctuations for different $n$ are generally correlated because the same
Brownian path drives all frequencies.  Thus the high-energy spectrum is better
described as a small correlated stochastic perturbation of the deterministic
lattice $(n\pi/L)_{n\ge1}$, rather than as a Poisson or L\'evy process.

Similarly, the eigenfunction $u_k(x)$ is not a martingale in $x$.  It
satisfies
\[
    \dd Y_x=P_x\,\dd x,\qquad
    \dd P_x=-k^2Y_x\,\dd x+\rho Y_x\,\dd B_x,
\]
so both components contain drift terms.  The first-order correction
\[
    I_k(x)=\int_0^x\sin(k(x-s))\sin(ks)\,\dd B_s
\]
is a Volterra stochastic integral rather than a martingale.  If
\[
    J_k(x)=\int_0^x\cos(k(x-s))\sin(ks)\,\dd B_s,
\]
then
\[
    \dd I_k(x)=kJ_k(x)\,\dd x,\qquad
    \dd J_k(x)=\sin(kx)\,\dd B_x-kI_k(x)\,\dd x.
\]
Thus the fluctuation of the eigenfunction is martingale-driven, but the
eigenfunction itself is not a martingale.
\end{remark}

Natural continuations of the present work include a full treatment of general separated boundary conditions, including Neumann and Robin endpoints, colored-noise approximations converging to the white-noise model, and higher-order stochastic asymptotics. Another possible direction is to use the high-mode eigenfunction expansions obtained here in random spectral-decomposition problems on a fixed interval.


\begin{thebibliography}{99}

\bibitem{BalGuReview}
G. Bal and Y. Gu,
\newblock Limiting models for equations with large random potential; a review,
\newblock \emph{Commun. Math. Sci.} \textbf{13} (2015), no. 3, 729--748.

\bibitem{Billingsley1999}
P. Billingsley,
\newblock \emph{Convergence of Probability Measures}, second edition,
\newblock Wiley Series in Probability and Statistics, John Wiley \& Sons, New York, 1999.

\bibitem{CarmonaLacroix}
R. Carmona and J. Lacroix,
\newblock \emph{Spectral Theory of Random Schr\"odinger Operators},
\newblock Birkh\"auser, 1990.

\bibitem{Davydov1968}
Yu. A. Davydov,
\newblock Convergence of distributions generated by stationary stochastic processes,
\newblock \emph{Theory Probab. Appl.} \textbf{13} (1968), no. 4, 691--696.

\bibitem{DoukhanMassartRio1994}
P. Doukhan, P. Massart, and E. Rio,
\newblock The functional central limit theorem for strongly mixing processes,
\newblock \emph{Ann. Inst. H. Poincar\'e Probab. Statist.} \textbf{30} (1994), no. 1, 63--82.

\bibitem{DuYuan2006}
L. Du and X. Yuan,
\newblock Invariant tori of nonlinear Schr\"odinger equations with a given potential,
\newblock \emph{Dyn. Partial Differ. Equ.} \textbf{3} (2006), no. 4, 331--346.

\bibitem{DumazLabbe2020}
L. Dumaz and C. Labb\'{e},
\newblock Localization of the continuous Anderson Hamiltonian in 1-D,
\newblock \emph{Probab. Theory Related Fields} \textbf{176} (2020), 353--419.

\bibitem{DumazLabbe2023}
L. Dumaz and C. Labb\'{e},
\newblock The delocalized phase of the Anderson Hamiltonian in 1-D,
\newblock \emph{Ann. Probab.} \textbf{51} (2023), 805--839.

\bibitem{DumazLabbe2024Crossover}
L. Dumaz and C. Labb\'{e},
\newblock Localization crossover for the continuous Anderson Hamiltonian in 1-D,
\newblock \emph{Invent. Math.} \textbf{235} (2024), 345--440.

\bibitem{DumazLabbe2024}
L. Dumaz and C. Labb\'{e},
\newblock Anderson localization for the 1-D Schr\"{o}dinger operator with white noise potential,
\newblock \emph{J. Funct. Anal.} \textbf{286} (2024), no. 1, 110191.

\bibitem{EGNT}
J. Eckhardt, F. Gesztesy, R. Nichols, and G. Teschl,
\newblock Weyl--Titchmarsh theory for Sturm--Liouville operators with distributional potentials,
\newblock \emph{Opuscula Math.} \textbf{33} (2013), 467--563.

\bibitem{FukushimaNakao}
M. Fukushima and S. Nakao,
\newblock On spectra of the Schr\"odinger operator with a white Gaussian noise potential,
\newblock \emph{Z. Wahrscheinlichkeitstheorie verw. Gebiete} \textbf{37} (1977), 267--274.

\bibitem{HrynivMykytyuk2006}
R. O. Hryniv and Ya. V. Mykytyuk,
\newblock Inverse spectral problems for Sturm--Liouville operators with singular potentials. IV. Potentials in the Sobolev space scale,
\newblock \emph{Proc. Edinburgh Math. Soc.} \textbf{49} (2006), 309--329.

\bibitem{HrynivMykytyukJFA2006}
R. O. Hryniv and Ya. V. Mykytyuk,
\newblock Eigenvalue asymptotics for Sturm--Liouville operators with singular potentials,
\newblock \emph{J. Funct. Anal.} \textbf{238} (2006), no. 1, 27--57.

\bibitem{HrynivMykytyuk2012}
R. O. Hryniv and Ya. V. Mykytyuk,
\newblock Self-adjointness of Schr\"odinger operators with singular potentials,
\newblock \emph{Methods Funct. Anal. Topology} \textbf{18} (2012), no. 2, 152--159.

\bibitem{KampfSpodarev2018}
J. Kampf and E. Spodarev,
\newblock A functional central limit theorem for integrals of stationary mixing random fields,
\newblock \emph{Theory Probab. Appl.} \textbf{63} (2018), no. 1, 135--150.

\bibitem{Kirsch2008}
 W. Kirsch,
  \newblock An Invitation to Random Schr{\"o}dinger Operators,
  \newblock \emph{Panoramas et Synth{\`e}ses} \textbf{25} (2008), 1--119.

\bibitem{Kuksin1993}
S. B. Kuksin,
\newblock \emph{Nearly Integrable Infinite-Dimensional Hamiltonian Systems},
\newblock Lecture Notes in Mathematics, Vol. 1556, Springer, Berlin, 1993.

\bibitem{KuksinPoeschel1996}
S. B. Kuksin and J. P\"oschel,
\newblock Invariant Cantor manifolds of quasi-periodic oscillations for a nonlinear Schr\"odinger equation,
\newblock \emph{Ann. of Math.} \textbf{143} (1996), no. 1, 149--179.

\bibitem{Minami}
N. Minami,
\newblock Definition and self-adjointness of the stochastic Airy operator,
\newblock \emph{Markov Process. Relat. Fields} \textbf{21} (2015), no. 3, 695--711.

\bibitem{PasturFigotin}
L. Pastur and A. Figotin,
\newblock \emph{Spectra of Random and Almost-Periodic Operators},
\newblock Springer, 1992.

\bibitem{Peng2000}
S. Peng,
\newblock Problem of eigenvalues of stochastic Hamiltonian systems with boundary conditions,
\newblock \emph{Stochastic Process. Appl.} \textbf{88} (2000), no. 2, 259--290.

\bibitem{RRV}
J. A. Ram\'irez, B. Rider, and B. Vir\'ag,
\newblock Beta ensembles, stochastic Airy spectrum, and a diffusion,
\newblock \emph{J. Amer. Math. Soc.} \textbf{24} (2011), no. 4, 919--944.

\bibitem{SavchukShkalikov}
A. M. Savchuk and A. A. Shkalikov,
\newblock Sturm--Liouville operators with distribution potentials,
\newblock \emph{Trans. Moscow Math. Soc.} \textbf{2003}, 143--192.

\bibitem{Thompson1983}
M. Thompson,
\newblock The state density for second order ordinary differential equations with white Gaussian noise potential,
\newblock \emph{Boll. Un. Mat. Ital. B} \textbf{2} (1983), 283--296.

\bibitem{Yuan2006WavePotential}
X. Yuan,
\newblock Quasi-periodic solutions of nonlinear wave equations with a prescribed potential,
\newblock \emph{Discrete Contin. Dyn. Syst.} \textbf{16} (2006), no. 3, 615--634.

\end{thebibliography}
\end{document}